\documentclass[twoside,12pt,leqno]{article}
\advance\topmargin by -\headheight\advance\topmargin by -\headsep
\newcommand{\tmvolxx}{xx}
\newcommand{\tmyearyyyy}{yyyy}

\newcommand{\FirstPageHead}[3]{
{\footnotesize 
\vskip -8mm 
\centerline {Travaux math\'ematiques, \quad 
Volume #1 (#2), 
#3,\quad \copyright\  Universit\'e du Luxembourg}}\vspace{-3mm}}

\usepackage{amsmath}
\usepackage{amsthm}
\usepackage{amssymb}
\usepackage{amscd}
\usepackage{graphicx}
\usepackage{epsfig}
\usepackage[all]{xy}
\usepackage{mathtools}
\usepackage{multirow}
\usepackage{enumerate}
\numberwithin{equation}{section}
\newtheorem{theorem}{Theorem}[section]
\newtheorem{lemma}[theorem]{Lemma}
\newtheorem{proposition}[theorem]{Proposition}
\newtheorem{corollary}[theorem]{Corollary}

\theoremstyle{definition}
\newtheorem{definition}[theorem]{Definition}
\newtheorem{example}[theorem] {Example}

\newtheorem{remark}[theorem]{Remark}

\numberwithin{equation}{section}

\newcommand{\N}{\mathbb{N}}

\newcommand{\R}{\mathbb{R}}
\newcommand{\C}{\mathbb{C}}
\newcommand{\K}{\mathbb{K}}
\newcommand{\as}{\mathfrak{as}}

\newcommand{\auteur}[1]{{\sc #1}}

\usepackage[pdftex,%
bookmarks=true,%
bookmarksnumbered=true,%
plainpages=false,%
breaklinks=true,%
]{hyperref}

\allowdisplaybreaks
\begin{document}
\thispagestyle{empty}
\FirstPageHead{\tmvolxx}{\tmyearyyyy}{\pageref{firstpage}--\pageref{lastpage}}
\label{firstpage}




\markboth{M. Bordemann, O. Elchinger and A. Makhlouf}{Twisting Poisson algebras, coPoisson algebras and Quantization }
$ $
\bigskip

\bigskip

\centerline{{\Large   Twisting Poisson algebras, coPoisson algebras}}

\

 \centerline{{\Large   and Quantization}}

\bigskip
\bigskip
\centerline{{\large by  Martin Bordemann, Olivier Elchinger }}

\

\centerline{{\large and Abdenacer Makhlouf}}

\vspace*{.7cm}

\begin{abstract}
The purpose of this paper is to study twistings of Poisson algebras or bialgebras, coPoisson algebras or bialgebras and star-products. We consider Hom-algebraic structures generalizing classical algebraic structures by twisting the identities by a linear self map. We summarize the  results on Hom-Poisson algebras and introduce  Hom-coPoisson algebras and bialgebras. We show that there exists a duality between Hom-Poisson bialgebras and Hom-coPoisson bialgebras. A relationship between enveloping Hom-algebras endowed with Hom-coPoisson structures and corresponding Hom-Lie bialgebra structures  is studied.   Moreover we set  quantization problems and  generalize the notion of star-product.  In particular, we characterize the twists for the Moyal-Weyl product for polynomials of several variables.
\end{abstract}

\pagestyle{myheadings}

\section*{Introduction}
The study of nonassociative algebras was originally motivated by
certain problems in physics and other branches of mathematics. The
first motivation to study nonassociative Hom-algebras comes from
quasi-deformations of Lie algebras of vector fields, in particular
$q$-deformations of Witt and Virasoro algebras
\cite{AizawaSaito, ChaiIsKuLuk,
CurtrZachos1, Kassel1,
Hu}.

Hom-Lie algebras were first introduced by Hartwig, Larsson and Silvestrov  in order to describe $q$-deformations of Witt and Virasoro algebras using $\sigma$-derivations (see \cite{HLS}). The corresponding associative type objects, called Hom-associative algebras were introduced by Makhlouf and Silvestrov in \cite{MakSil-def}. The enveloping algebras of Hom-Lie algebras were studied by Yau in \cite{Yau:EnvLieAlg}. The dual notions, Hom-coalgebras, Hom-bialgebras, Hom-Hopf algebras and Hom-Lie coalgebras,  were studied first in \cite{HomHopf,HomAlgHomCoalg}. The study was  enhanced  in \cite{Yau:YangBaxter,Yau:comodule} and done with a categorical point of view in \cite{Canepl2009}.  Further developments and results could be found in \cite{AEM,AmmarMakhloufJA2010,FregierGohrSilv,Gohr,Yau:homology,Yau:YangBaxter2}.
The notion of Hom-Poisson algebras appeared first in \cite{MakSil-def} and then studied in \cite{Yau:Poisson}. The main feature of Hom-algebra and Hom-coalgebra structures is that the classical identities are twisted by a  linear self map.

In this paper, we review the results on Hom-Poisson algebras and introduce the notions of Hom-coPoisson algebra and Hom-coPoisson bialgebra. We show that there is a duality between the Hom-coPoisson bialgebras and Hom-Poisson bialgebras, generalizing to Hom-setting the result in \cite{OhPark}. Moreover we set the quantization problems and  generalize the notion of star-product.  In particular, we characterize the twists for the Moyal-Weyl product for polynomials of several variables.
In Section \ref{twist-usual}, we summarize the definitions of the Hom-algebra and Hom-coalgebra structures twisting  classical structures of associative algebras, Lie algebras and their duals. We extend to Hom-structures  the classical relationship between the algebra and its finite dual and provide an easy tool to obtain a Hom-structure from a classical structure and provide some examples. In Section \ref{Hom-Poisson}, we review the main results on Hom-Poisson algebras from \cite{MakSil-def,Yau:Poisson} and construct some new examples. In Section \ref{Hom-coPoisson}, we introduce and study Hom-coPoisson algebras, Hom-coPoisson bialgebras and Hom-coPoisson Hopf algebras. We provide some key constructions and show a link between enveloping algebras (viewed as a Hom-bialgebra) endowed with a coPoisson structure and Hom-Lie bialgebras. Moreover, we show that there is a duality between Hom-Poisson bialgebra (resp. Hom-Poisson Hopf algebra) and Hom-coPoisson bialgebra (resp. Hom-coPoisson Hopf algebra). In Section \ref{def-Poisson}, we review the 1-parameter formal deformation of Hom-algebras (resp. Hom-coalgebras) and  their relationship to Hom-Poisson algebra (resp. Hom-coPoisson algebras). Then we state the quantization problem and define star-product in Hom-setting. Finally we study twistings of Moyal-Weyl star-product.

\section{Twisting usual structures: Hom-algebras and Hom-coalgebras} \label{twist-usual}

Throughout  this paper $\K$ denotes a field of characteristic $0$ and $A$ is a $\K$-module. In the sequel we denote by $\sigma$ the linear map $\sigma :A^{\otimes 3}\rightarrow A^{\otimes 3}$,   defined as $\sigma(x_1 \otimes x_2 \otimes x_3) = x_3 \otimes x_1 \otimes x_2$ and by $\tau_{ij}$  linear maps $\tau :A^{\otimes n}\rightarrow A^{\otimes n}$ 
where $\tau_{i j}(x_1\otimes\cdots\otimes x_i\otimes \cdots \otimes x_j \otimes \cdots\otimes x_n) = (x_1\otimes\cdots\otimes x_j\otimes \cdots \otimes x_i \otimes \cdots\otimes x_n)$.\\

We mean by a Hom-algebra  a triple $(A, \mu, \alpha)$ in which   $\mu : A^{\otimes 2} \to A$ is a linear map and $\alpha : A \to A$ is a linear self map. The linear map  $\mu^{op} : A^{\otimes 2} \to A$ denotes the opposite map, i.e. $\mu^{op} = \mu\circ\tau_{12} $.  A  Hom-coalgebra is a triple $(A, \Delta, \alpha)$ in which   $\Delta :A  \to A^{\otimes 2}$ is a linear map and $\alpha : A \to A$ is a linear self map. The linear map  $\Delta^{op} : A  \to A^{\otimes 2}$ denotes the opposite map, i.e. $\Delta^{op} = \tau_{12} \circ \Delta$.
For a linear self-map $\alpha : A \to A$, we denote by $\alpha^n$ the $n$-fold composition of $n$ copies of $\alpha$, with $\alpha^0 \cong Id$. A Hom-algebra $(A, \mu, \alpha)$ (resp. a Hom-coalgebra $(A, \Delta, \alpha)$) is said to be \emph{multiplicative}  if $\alpha \circ \mu = \mu \circ \alpha^{\otimes 2}$ (resp. $\alpha ^{\otimes 2}\circ \Delta = \Delta \circ \alpha$). The Hom-algebra is called \emph{commutative} if $\mu = \mu^{op}$ and the Hom-coalgebra is called \emph{cocommutative} if $\Delta = \Delta^{op}$.

Unless otherwise specified, we assume in this paper that the Hom-algebras and Hom-coalgebras are multiplicative.  Classical algebras or  coalgebras are also regarded as a Hom-algebras or Hom-coalgebras with identity twisting map. Given a Hom-algebra $(A, \mu, \alpha)$, we often use the abbreviation $\mu(x, y) = xy$ for $x, y \in A$. Likewise, for a Hom-coalgebra $(A,\Delta,\alpha)$, we will use Sweedler's notation $\Delta(x) = \sum_{(x)} x_1 \otimes x_2$ but often omit the symbol of summation. When the Hom-algebra (resp. Hom-coalgebra) is multiplicative, we also say that $\alpha$ is multiplicative for $\mu$ (resp.  $\Delta$). \\

In the following, we summarize the definitions and properties of usual Hom-structures, provide some examples and recall the twisting principle giving a way to turn a classical structure to Hom-structure using algebra morphisms.

\subsection{Hom-associative algebras and Hom-Lie algebras}
We recall the definitions of Hom-Lie algebras and Hom-associative algebras.
\begin{definition}
Let $(A, \mu, \alpha)$ be a Hom-algebra.
\begin{enumerate}
\item The \emph{Hom-associator} $\as_{\mu,\alpha} : A^{\otimes 3} \to A$ is defined as
\begin{equation}
\as_{\mu,\alpha} = \mu \circ (\mu \otimes \alpha - \alpha \otimes \mu).
\end{equation}
\item The Hom-algebra $A$ is called a \emph{Hom-associative algebra} if it satisfies the Hom-associative identity
\begin{equation}
\as_{\mu,\alpha} = 0.
\end{equation}
\item A Hom--associative algebra $A$ is called  \emph{unital} if there exists a linear map $\eta: \K\to A$  such that 
\begin{equation}
\mu\circ\left(  Id_{A}\otimes\eta\right)  =\mu\circ\left(  \eta\otimes
Id_{A}\right)  =\alpha.
\end{equation}
The unit element is $1_A=\eta(1)$.
\item The \emph{Hom-Jacobian} $J_{\mu,\alpha} : A^{\otimes 3} \to A$ is defined as
\begin{equation}
J_{\mu,\alpha} = \mu \circ (\alpha \otimes \mu) \circ (Id + \sigma + \sigma^2).
\end{equation}
\item The Hom-algebra $A$ is called a \emph{Hom-Lie algebra} if it satisfies $\mu + \mu^{op} = 0$ and the Hom-Jacobi identity
\begin{equation}
J_{\mu,\alpha} = 0.
\end{equation}
\end{enumerate}
\end{definition}
We recover the usual definitions of the associator, an associative algebra, the Jacobian, and a Lie algebra when the twisting map $\alpha$ is the identity map.
In terms of elements $x, y, z \in A$, the Hom-associator and the Hom-Jacobian are
\begin{align*}
\as_{\mu,\alpha}(x, y, z) &=\ (xy)\alpha(z) - \alpha(x)(yz), \\
J_{\mu,\alpha}(x, y, z) &=\ \circlearrowleft_{x,y,z}{[\alpha(x),[y,z]]},
\end{align*}
where the bracket denotes the product and  $\circlearrowleft_{x,y,z}$ denotes the cyclic sum  on $x,y,z$.

Let $\left( A,\mu,\alpha \right)$ and $A^{\prime} = \left( A^{\prime},\mu^{\prime}, \alpha^{\prime}\right)$ be two Hom-algebras (either Hom-associative or Hom-Lie). A linear map $f\ : A \rightarrow A^{\prime}$ is a \emph{morphism of Hom-algebras} if%
\[
\mu^{\prime} \circ (f \otimes f) = f \circ \mu \quad \text{ and } \qquad f \circ \alpha = \alpha^{\prime}\circ f.
\]

It is said to be a \emph{weak morphism} if holds only the first condition.

\begin{remark}
It was shown in \cite{MakSil-struct} that the commutator of a Hom-associative algebra leads to a Hom-Lie algebra. The construction of the enveloping algebras of Hom-Lie algebras is described in \cite{Yau:EnvLieAlg}.
\end{remark}

\begin{example}\label{example1ass}
Let $\{x_1,x_2,x_3\}$  be a basis of a $3$-dimensional vector space
$A$ over $\K$. The following multiplication $\mu$ and linear map
$\alpha$ on $A=\K^3$ define a Hom-associative algebra over $\K${\rm:}
\[
\begin{array}{ll}
  \begin{array}{lll}
  \mu ( x_1,x_1)&=& a x_1, \ \\
  \mu ( x_1,x_2)&=&\mu ( x_2,x_1)=a x_2,\\
  \mu ( x_1,x_3)&=&\mu ( x_3,x_1)=b x_3,\\
  \end{array}
 & \quad
  \begin{array}{lll}
  \mu ( x_2,x_2)&=& a x_2, \ \\
  \mu ( x_2, x_3)&=& b x_3, \ \\
  \mu ( x_3,x_2)&=& \mu ( x_3,x_3)=0,
  \end{array}
\end{array}
\]

\[
\alpha(x_1) = a x_1, \quad
\alpha(x_2) = a x_2, \quad
\alpha(x_3) = b x_3
\]
where $a,b$ are parameters in $\K$. This algebra is not associative when $a\neq b$ and $b\neq 0$, since
\[
\mu(\mu(x_1,x_1),x_3))- \mu(x_1,\mu(x_1,x_3))=(a-b)b x_3.
\]
\end{example}

\begin{example}[Jackson $\mathfrak{sl}_2$]\label{JacksonSL2}
The Jackson $\mathfrak{sl}_2$ is a $t$-deformation of the classical Lie algebra  $\mathfrak{sl}_2$. It carries a  Hom-Lie algebra structure but not a Lie algebra structure by using Jackson derivations. It is defined with respect to a basis $\{x_1,x_2,x_3\}$  by the brackets and a linear  map $\alpha$ such that
\[
\begin{array}{cc}
  \begin{array}{ccc}
  [ x_1, x_2 ] &= &2 x_2, \\ {}
  [x_1, x_3 ]&=& -2(1+t) x_3, \\ {}
  [ x_2,x_3 ] & = &  \left(1+\dfrac{t}{2}\right) x_1,
  \end{array}
 & \quad
  \begin{array}{ccc}
  \alpha (x_1)&=& x_1, \\
  \alpha (x_2)&=& \dfrac{2+t}{2(1+t)} x_2, \\
  \alpha (x_3)&=& \left(1+\dfrac{t}{2}\right) x_3,
  \end{array}
\end{array}
\]
where $t$ is a parameter in $\K$. If $t=0$ we recover the classical $\mathfrak{sl}_2$.
\end{example}

The following proposition gives an easy way to twist classical structures in Hom-structures.

\begin{theorem}[\cite{Yau:homology}] \hfill \label{twist-ass_Lie}
\begin{enumerate}[(1)]
\item Let $A = (A, \mu)$ be an associative algebra and $\alpha : A \to A$ be a linear map which is multiplicative with respect to $\mu$, i.e. $\alpha \circ \mu = \mu \circ \alpha^{\otimes 2}$. Then $A_\alpha = (A, \mu_\alpha = \alpha \circ \mu, \alpha)$ is a Hom-associative algebra.
\item Let $A = (A, [~,~])$ be a Lie algebra and $\alpha : A \to A$ be a linear map which is multiplicative with respect to $[~,~]$, i.e. $\alpha \circ [~,~] = [~,~] \circ \alpha^{\otimes 2}$. Then $A_\alpha = (A, [~,~]_\alpha = \alpha \circ [~,~], \alpha)$ is a Hom-Lie algebra.
\end{enumerate}
\end{theorem}

\begin{proof}
\begin{enumerate}[(1)]
\item We have
\begin{equation*}
\begin{split}
\as_{\mu_\alpha,\alpha} &= \mu_\alpha \circ (\mu_\alpha \otimes \alpha - \alpha \otimes \mu_\alpha) \\
&= \alpha \circ \mu \circ \alpha^{\otimes 2} \circ (\mu \otimes Id - Id \otimes \mu) \\
&= \alpha^2 \circ \as_{\mu, Id} = 0
\end{split}
\end{equation*}
since $(A,\mu)$ is associative, so $A_\alpha$ is a Hom-associative algebra.
\item We have $\forall\ x,y \in A, \ [y,x] = -[x,y]$ and
\begin{equation*}
\begin{split}
J_{[~,~]_\alpha,\alpha} &= [~,~]_\alpha \circ (\alpha \otimes [~,~]_\alpha) \circ (Id + \sigma + \sigma^2) \\
&= \alpha \circ [~,~] \circ \alpha^{\otimes 2} \circ (Id \otimes [~,~]) \circ (Id + \sigma + \sigma^2) \\
&= \alpha^2 \circ J_{[~,~],Id} = 0
\end{split}
\end{equation*}
since $[~,~]$ is a Lie bracket, so $A_\alpha$ is a Hom-Lie algebra.
\end{enumerate}
\end{proof}

\begin{remark}
More generally, the categories of Hom-associative algebras and Hom-Lie algebras   are closed under twisting self-weak morphisms. If  $A= (A, \mu, \alpha)$ is a Hom-associative algebra  (resp. Hom-Lie algebra) and $\beta$ a weak morphism, then $A_\beta = (A, \mu_\beta = \beta \circ \mu, \beta\circ\alpha)$ is a Hom-associative algebra (resp. Hom-Lie algebra) as well (see \cite{Yau:Poisson}).
\end{remark}

\begin{example}
To twist the usual Lie algebra $\mathfrak{sl}_2$, generated by the elements $e,f,h$ and relations $[h,e]=2e, \ [h,f]=-2f, \ [e,f]=h$, we figure out morphisms $\alpha : \mathfrak{sl}_2 \to \mathfrak{sl}_2$ written as $\alpha = (\alpha_{ij})$ with respect to the basis $(e,f,h)$, by solving  the equations $\alpha([x,y]) = [\alpha(x),\alpha(y)]$ with respect to the  basis in the coefficients $\alpha_{ij}$. We obtain $({\mathfrak{sl}_2})_\alpha$, Hom-Lie versions of $\mathfrak{sl}_2$, with $\alpha$ given by
\begin{enumerate}
\item $\alpha(e) = \lambda e, \ \alpha(f) = -\lambda \mu^2 e + \lambda^{-1} f + \mu h, \ \alpha(h) = -2\lambda \mu e + h$,
\item $\alpha(e) = - \dfrac{\lambda \mu^2}{4} e + \lambda f + \dfrac{\lambda \mu}{2} h, \ \alpha(f) = \lambda^{-1} e, \ \alpha(h) =\mu e - h$,
\item
\(
\begin{aligned}[t]
\alpha(e) &= \frac{1}{4} \lambda (\mu + 1)^2 e -\frac{1}{4} \lambda \nu^2  f - \frac{1}{4} \lambda (\mu + 1) \nu h, \\
\alpha(f) &= \lambda^{-1} (\mu - 1)^2 \nu^{-2} e - \lambda^{-1} f + \lambda^{-1} (\mu - 1) \nu^{-1} h, \\
\alpha(h) &= (1-\mu^2) \nu^{-1} e + \nu f + \mu h,
\end{aligned}
\)
\end{enumerate}
where  $\lambda, \nu$ are nonzero elements in $\K$ and $\mu$ is an element in $\K$.
\end{example}

\subsection{Hom-coalgebras, Hom-bialgebras and Hom-Hopf algebras}
In this section we summarize and describe some of basic
properties of Hom-coalgebras, Hom-bialgebras and Hom-Hopf algebras
which generalize the classical coalgebra, bialgebra and Hopf algebra
structures.

\begin{definition}
A \emph{Hom-coalgebra} is a
triple $\left( A,\Delta ,\alpha \right) $ where $A$ is a $\K$-module and
$\Delta : A \rightarrow A \otimes A,\ \alpha : A \rightarrow A$ are linear maps.

A \emph{Hom-coassociative coalgebra} is a Hom-coalgebra satisfying
\begin{equation}\label{C1}
\left(\alpha \otimes \Delta \right) \circ \Delta = \left( \Delta \otimes \alpha\right) \circ \Delta.
\end{equation}
A Hom-coassociative coalgebra is said to be \emph{counital} if there exists a map $\varepsilon :A\rightarrow \K$ satisfying
\begin{equation}\label{C2}
\left( id \otimes \varepsilon \right) \circ \Delta = \alpha \qquad \text{ and }\qquad
\left( \varepsilon \otimes id\right) \circ \Delta = \alpha.
\end{equation}
\end{definition}

Let $\left( A,\Delta ,\alpha \right) $
and $\left( A^{\prime },\Delta^{\prime
},\alpha^{\prime }\right)
$ be two Hom-coalgebras (resp.  Hom-coassociative coalgebras). A
linear map  $f\ :A\rightarrow A^{\prime }$ is a \emph{morphism of  Hom-coalgebras} (resp. \emph{Hom-coassociative coalgebras}) if%
\[
(f \otimes f)\circ \Delta = \Delta^{\prime} \circ f
\quad \text{ and} \quad f \circ\alpha = \alpha^{\prime}\circ f.
\]
If furthermore the Hom-coassociative coalgebras admit counits $\varepsilon$ and $\varepsilon^{\prime}$, we have moreover $\varepsilon = \varepsilon^{\prime}\circ f $.

The following theorem shows how to construct a new Hom-coassociative Hom-coalgebra starting with a Hom-coassociative Hom-coalgebra and a Hom-coalgebra morphism. We only need the coassociative comultiplication of the coalgebra. 

\begin{theorem} \label{thmConstrHomCoalg}
Let $(A,\Delta,\alpha,\varepsilon)$ be a counital Hom-coalgebra  and $\beta : A\rightarrow A$ be a weak Hom-coalgebra morphism. Then $(A,\Delta_\beta,\alpha \circ \beta,\varepsilon)$, where $\Delta_\beta = \Delta \circ \beta$,  is a counital Hom-coassociative coalgebra.
\end{theorem}

\begin{proof}
We show that $(A,\Delta_\beta,\alpha\circ\beta)$ satisfies the axiom \eqref{C1}.

Indeed, using the fact that $(\beta \otimes \beta) \circ \Delta = \Delta \circ \beta$, we have
\begin{align*}
\left( \alpha\circ\beta\otimes \Delta_\beta \right)
\circ \Delta_\beta &= \left(\alpha\circ\beta \otimes \Delta\circ\beta\right) \circ \Delta\circ\beta \\
&=(\left( \alpha \otimes
\Delta\right) \circ\Delta)\circ \beta^2\\
&=(\left( \Delta\otimes\alpha
\right) \circ\Delta)\circ \beta^2\\
&=\left(\Delta_\beta \otimes\alpha\circ\beta
\right) \circ \Delta_\beta.
\end{align*}
Moreover, the axiom \eqref{C2} is also satisfied, since we have
\begin{equation*}
\left( id \otimes \varepsilon \right) \circ \Delta_\beta = \left( id \otimes \varepsilon \right) \circ \Delta \circ \beta = \alpha \circ \beta = \left( \varepsilon \otimes id\right) \circ \Delta \circ \beta = \left( \varepsilon \otimes id\right) \circ \Delta_\beta.
\end{equation*}
\end{proof}

\begin{remark}
The previous theorem shows that the category of coassociative Hom-coalgebras is closed under weak Hom-coalgebra morphisms. It leads to the following examples:
\begin{enumerate}
\item Let $(A,\Delta)$ be a coassociative coalgebra and $\beta : A \to A$ be a coalgebra morphism. Then $(A,\Delta_\beta,\beta)$ is a Hom-coassociative coalgebra.
\item Let $(A,\Delta,\alpha)$ be a (multiplicative) coassociative Hom-coalgebra. For all non negative integer $n$,  $(A,\Delta_{\alpha^n}, \alpha^{n+1})$ is a Hom-coassociative coalgebra.
\end{enumerate}

\end{remark}

In the following we show that there is a duality between  Hom-associative and the Hom-coassociative structures (see \cite{HomHopf,HomAlgHomCoalg}).

\begin{theorem}
Let $(A,\Delta,\alpha)$ be a Hom-coalgebra. Then $(A^*,\Delta^*,\alpha^*)$ is an Hom-algebra, which is Hom-associative if and only if $A$ is Hom-coassociative, and unital if and only if $A$ is counital.
\end{theorem}

\begin{proof}
The product $\mu = \Delta^*$ is defined from $A^* \otimes A^*$ to $A^*$ by
\begin{equation*}
(fg)(x) = \Delta^*(f,g)(x) = \langle \Delta(x),f \otimes g \rangle = (f \otimes g)(\Delta(x)) = \sum_{(x)} f(x_1)g(x_2), \quad \forall x \in A
\end{equation*}
where $\langle \cdot,\cdot \rangle$ is the natural pairing between the vector space $A \otimes A$ and its dual vector space. For $f,g,h \in A^*$ and $x \in A$, we have
\begin{gather*}
(fg) \alpha^*(h)(x) = \langle (\Delta \otimes \alpha) \circ \Delta(x),f \otimes g \otimes h \rangle
\shortintertext{and}
\alpha^*(f)(gh)(x) = \langle (\alpha \otimes \Delta) \circ \Delta(x),f \otimes g \otimes h \rangle
\end{gather*}
so the Hom-associativity $\mu \circ (\mu \otimes \alpha^* - \alpha^* \otimes \mu) = 0$ is equivalent to the Hom-coassociativity $(\Delta \otimes \alpha - \alpha \otimes \Delta) \circ \Delta = 0$. \\
Moreover, if $A$ has a counit $\varepsilon$ satisfying $(id \otimes \varepsilon) \circ \Delta = \alpha = (\varepsilon \otimes id) \circ \Delta$ then for $f \in A^*$ and $x \in A$ we have
\begin{gather*}
(\varepsilon f)(x) = \sum_{(x)} \varepsilon(x_1) f(x_2) = \sum_{(x)} f(\varepsilon(x_1) x_2) = f(\alpha(x)) = \alpha^*(f)(x) \\
\shortintertext{and}
(f \varepsilon)(x) = \sum_{(x)} f(x_1) \varepsilon(x_2) = \sum_{(x)} f(x_1 \varepsilon(x_2)) = f(\alpha(x)) = \alpha^*(f)(x)
\end{gather*}
which shows that $\varepsilon$ is the unit in $A^*$.
\end{proof}

The dual of a Hom-algebra $(A,\mu,\alpha)$ is not always a Hom-coalgebra, because the coproduct does not land in the good space: $\mu^* : A^* \to (A \otimes A)^* \supsetneq A^* \otimes A^*$. It is the case if the Hom-algebra is of finite dimension, since $(A\otimes A)^* = A^* \otimes A^*$.

In the general case, for any algebra $A$, define
\begin{equation*}
A^\circ = \{f \in A^*, f(I)=0\ \text{for some cofinite ideal $I$ of $A$}\},
\end{equation*}
where a \emph{cofinite ideal} $I$ is an ideal $I \subset A$ such that $A/I$ is finite dimensional.

$A^\circ$ is a subspace of $A^*$ since it is closed under multiplication by scalars and the sum of two elements of $A^\circ$ is again in $A^\circ$ since the intersection of two cofinite ideals is again such. If $A$ is finite dimensional, of course $A^\circ = A^*$.

For a map $f : A \to B$ we note $f^\circ := f^*|_{B\circ} : B^\circ \to A^\circ$. Since $A^\circ \otimes A^\circ = (A \otimes A)^\circ$ (see \cite[Lemma 6.0.1]{Sweedler}) the dual $\mu^* : A^* \to (A \otimes A)^*$ of the multiplication $\mu : A \otimes A$ satisfies $\mu^*(A^\circ)  \subset A^\circ \otimes A^\circ$. Indeed, for $f \in A^*,\ x,y \in A$, we have $\langle \mu^*(f),x \otimes y \rangle = \langle f,xy \rangle$. So if $I$ is a cofinite ideal such that $f(I) = 0$, then $I \otimes A + A \otimes I$ is a cofinite ideal of $A \otimes A$ which vanish on $\mu^*(f)$.

Define $\Delta = \mu^\circ = \mu^*|_{A^\circ}$ and $\varepsilon : A^\circ \to \K$ by $\varepsilon(f) = f(1)$.

\begin{theorem}
Let $(A,\mu,\alpha)$ be a Hom-algebra. Then $(A^\circ,\Delta,\alpha^\circ)$ is an Hom-coalgebra, which is Hom-coassociative if and only if $A$ is Hom-associative, and counital if and only if $A$ is unital.
\end{theorem}

\begin{proof}
The coproduct $\Delta$ is defined from $A^\circ$ to $A^\circ \otimes A^\circ$ by
\begin{equation*}
\Delta(f)(x \otimes y) = \mu^*|_{A^\circ}(f)(x \otimes y) = \langle \mu(x \otimes y),f \rangle = f(xy), \quad x,y \in A.
\end{equation*}
For $f,g,h \in A^\circ$ and $x,y \in A$, we have
\begin{gather*}
(\Delta \circ \alpha^\circ) \circ \Delta(f)(x \otimes y \otimes z) = \langle \mu \circ (\mu \otimes \alpha) (x \otimes y \otimes z),f \rangle
\shortintertext{and}
(\alpha^\circ \circ \Delta) \circ \Delta(f)(x \otimes y \otimes z) = \langle \mu \circ (\otimes \mu\alpha) (x \otimes y \otimes z),f \rangle
\end{gather*}
so the Hom-coassociativity $(\Delta \otimes \alpha^\circ - \alpha^\circ \otimes \Delta) \circ \Delta = 0$  is equivalent to the Hom-associativity $\mu \circ (\mu \otimes \alpha - \alpha \otimes \mu) = 0$. \\
Moreover, if $A$ has a unit $\eta$ satisfying $\mu \circ (id \otimes \eta) = \alpha = \mu \circ (\eta \otimes id)$ then for $f \in A^\circ$ and $x \in A$ we have
\begin{gather*}
(\varepsilon \otimes id) \circ \Delta(f)(x) = f(1.x) = f(\alpha(x)) = \alpha^\circ(f)(x) \\
\intertext{and}
(id \otimes \varepsilon) \circ \Delta(f)(x) = f(x.1) = f(\alpha(x)) = \alpha^\circ(f)(x)
\end{gather*}
which shows that $\varepsilon : A^\circ \to \K$, $f \mapsto f(1)$ is the counit in $A^\circ$.
\end{proof}

\begin{definition}
A \emph{Hom-bialgebra} is a tuple $\left( A,\mu ,\alpha,\eta ,\Delta ,\beta,\varepsilon \right)$ where
\begin{enumerate}
\item $\left( A,\mu ,\alpha,\eta \right)$
is a Hom-associative algebra with a unit  $\eta$.
\item  $\left( A,\Delta ,\beta,\varepsilon \right)$ is a Hom-coassociative coalgebra with a counit $\varepsilon$.
\item  The linear maps $\Delta $ and $\varepsilon $ are compatible with the multiplication  $\mu$ and the unit $\eta$, that is
\begin{eqnarray}
&& \Delta \left(e\right) = e \otimes e \qquad
\text{where}\ e=\eta\left( 1\right), \\
&& \Delta \left( \mu(x\otimes y)\right) = \Delta \left( x \right) \bullet \Delta \left( y\right)
=\sum_{(x)(y)}\mu(x_1 \otimes y_1) \otimes \mu( x_2 \otimes y_2), \\
&& \varepsilon \left( e\right) =1, \\
&& \varepsilon \left( \mu(x\otimes y)\right)
=\varepsilon \left(x\right) \varepsilon \left( y\right), \\
&& \varepsilon\circ \alpha \left( x\right)
=\varepsilon \left( x\right),
\end{eqnarray}
\end{enumerate}
where the bullet $\bullet$ denotes the
multiplication on tensor product.

If $\alpha = \beta$ the Hom-bialgebra is noted $(A,\mu,\eta,\Delta,\varepsilon,\alpha)$.
\end{definition}

Combining  previous observations, with only one twisting map, we obtain:
\begin{proposition}
Let  $\left( A,\mu, \Delta ,\alpha \right) $ be a Hom-bialgebra. Then the finite dual $\left( A^\circ ,\mu^\circ,  \Delta^\circ, \alpha^\circ  \right)$ is a Hom-bialgebra as well.
\end{proposition}

\begin{remark} \hfill
\begin{enumerate}
\item Given a Hom-bialgebra $\left( A,\mu ,\alpha,\eta ,\Delta,\beta ,\varepsilon \right)$, it is  shown in \cite{HomHopf,HomAlgHomCoalg} that the vector space $Hom \left( A,A \right)$ with the multiplication given
by the convolution product carries a structure of Hom-associative algebra.
\item An endomorphism $S$ of $A$ is said to be
an \emph{antipode} if it is  the inverse of the identity over $A$ for the Hom-associative algebra $Hom \left(A,A \right)$ with the multiplication given by the convolution product defined by
\[
f \star g= \mu \circ \left( f \otimes g \right) \Delta
\]
and the unit being $\eta \circ \varepsilon$.
\item A \emph{Hom-Hopf algebra} is a Hom-bialgebra with an antipode.
\end{enumerate}
\end{remark}

By combining Theorems \ref{twist-ass_Lie} and  \ref{thmConstrHomCoalg}, we obtain:
\begin{proposition}
Let $(A,\mu,\Delta,\alpha)$ be a Hom-bialgebra  and $\beta : A\rightarrow A$ be a Hom-bialgebra morphism commuting with $\alpha$. Then $(A,\mu_\beta,\Delta_\beta,\beta\circ \alpha)$ is a Hom-bialgebra.
\end{proposition}
Notice that the category of Hom-bialgebra is not closed under weak Hom-bialgebra morphisms.

\begin{example}[Universal enveloping Hom-algebra]
Given a multiplicative Hom-associative algebra $A = (A,\mu,\alpha)$, one can associate to it a multiplicative Hom-Lie algebra $HLie(A) = (A,[~,~],\alpha)$ with the same underlying module $(A,\alpha)$ and the bracket $[~,~] = \mu - \mu^{op}$.
This construction gives a functor $HLie$ from multiplicative Hom-associative algebras to multiplicative Hom-Lie algebras. In \cite{Yau:EnvLieAlg}, Yau constructed the left adjoint $U_{HLie}$ of $HLie$. He also made some minor modifications in \cite{Yau:comodule} to take into account the unital case.

The functor $U_{HLie}$ is defined as
\begin{equation}
U_{HLie}(A) = F_{HNAs}(A)/I^\infty \qquad \text{with} \quad F_{HNAs}(A) = \bigoplus_{n \geqslant 1} \bigoplus_{\tau \in T_n^{wt}} A_{\tau}^{\otimes n}
\end{equation}
for a multiplicative Hom-Lie algebra $(A,[~,~],\alpha)$. Here $T_n^{wt}$ is the set of weighted $n$-trees encoding the multiplication of elements (by trees) and twisting by $\alpha$ (by weights), $A_{\tau}^{\otimes n}$ is a copy of $A^{\otimes n}$ and $I^\infty$ is a certain submodule of relations build in such a way that the quotient is Hom-associative.

Moreover, the comultiplication $\Delta : U_{HLie}(A) \to U_{HLie}(A)~\otimes~U_{HLie}(A)$ defined by $\Delta(x) = \alpha(x) \otimes 1 + 1 \otimes \alpha(x)$ equips the multiplicative Hom-associative algebra $U_{HLie}(A)$ with a structure of Hom-bialgebra.
\end{example}

\subsection{Hom-Lie coalgebras and Hom-Lie bialgebras}

As it is the case for the Hom-associative algebras, the Hom-Lie algebras also have a dualized version, Hom-Lie coalgebras. They share the same kind of properties. We review here the principal results, similar results can be found in \cite{Yau:ClassicYangBaxter}.

\begin{definition}
A \emph{Hom-Lie coalgebra} $(A,\Delta,\alpha)$ is a (multiplicative) Hom-coalgebra satisfying $\Delta + \Delta^{op} = 0$ and the Hom-coJacobi identity
\begin{equation}
(Id + \sigma + \sigma^2) \circ (\alpha \otimes \Delta) \circ \Delta = 0.
\end{equation}
We call $\Delta$ the cobracket.
\end{definition}

We recover Lie coalgebra when $\alpha = Id$. Just like (co)associative (co)algebras we have the following dualization properties.

\begin{theorem} \hfill
\begin{enumerate}
\item Let $(A,\Delta,\alpha)$ be a Hom-Lie coalgebra. Then $(A^*,\Delta^*,\alpha^*)$ is an Hom-Lie algebra.
\item Let $(A,[~,~],\alpha)$ be a Hom-Lie algebra. Then $(A^\circ,[~,~]^\circ,\alpha^\circ)$ is an Hom-Lie coalgebra.
\end{enumerate}
\end{theorem}

The twisting principle also works, showing that the category of Hom-Lie coalgebras is closed under weak Hom-coalgebras morphisms.

\begin{theorem} \label{thmConstrHomLieCoalg}
Let $(A,\Delta,\alpha)$ be a Hom-Lie coalgebra and $\beta : A \to A$ a weak Hom-coalgebra morphism. Then $(A,\Delta_\beta = \Delta \circ \beta, \alpha \circ \beta)$ is a Hom-Lie coalgebra.
\end{theorem}

\begin{proof}
We have $\Delta_\beta + \Delta_\beta^{op} = (\Delta + \Delta^{op}) \circ \beta = 0$, and
\begin{align*}
(Id + \sigma + \sigma^2) \circ (\alpha \circ \beta \otimes \Delta_\beta) \circ \Delta_\beta &= (Id + \sigma + \sigma^2) \circ (\alpha \circ \beta \otimes \Delta \circ \beta) \circ \Delta \circ \beta \\
&= (Id + \sigma + \sigma^2) \circ (\alpha \otimes \Delta) \circ \Delta \circ \beta^2 \\
&= 0.
\end{align*}
\end{proof}

The previous theorem can be used to construct Hom-Lie coalgebras.

\begin{corollary} \hfill
\begin{enumerate}
\item Let $(A,\Delta)$ be a Lie coalgebra and $\beta : A \to A$ be a Lie coalgebra morphism. Then $(A,\Delta_\beta,\beta)$ is a Hom-Lie coalgebra.
\item Let $(A,\Delta,\alpha)$ be a (multiplicative) Hom-Lie coalgebra. For all non negative integer $n$, $(A,\Delta_{\alpha^n},\alpha^{n+1})$ is a Hom-Lie coalgebra.
\end{enumerate}
\end{corollary}

The Hom-Lie bialgebra structure was introduced first in \cite{Yau:ClassicYangBaxter}. The  definition presented below is slightly more  general. They border the class defined by Yau and permit to consider the compatibility condition for different $A$-valued cohomology of Hom-Lie algebras. 

\begin{definition}
A \emph{Hom-Lie bialgebra} is a tuple  $(A,[~,~],\alpha,\Delta,\beta)$ where
\begin{enumerate}
\item $(A,[~,~],\alpha)$ is a Hom-Lie algebra.
\item $(A,\Delta,\beta)$ is a Hom-Lie coalgebra.
\item The following compatibility condition holds for $x,y \in A$:
\begin{equation}\label{CompaCondHLieBialg}
\Delta([x,y]) = ad_{\alpha(x)}(\Delta(y)) - ad_{\alpha(y)}(\Delta(x)),
\end{equation}
where the adjoint map $ad_x : A^{\otimes n} \to A^{\otimes n} \ (n \geqslant 2)$ is given by
\begin{equation}
ad_x(y_1 \otimes \dotsb \otimes y_n) = \sum_{i=1}^n \beta(y_1) \otimes \dotsb \otimes \beta(y_{i-1}) \otimes [x,y_i] \otimes \beta(y_{i+1}) \otimes \dotsb \otimes \beta(y_n).
\end{equation}
\end{enumerate}
A \emph{morphism} $f : A \to A'$ of Hom-Lie bialgebras is a linear map commuting with $\alpha$ and $\beta$ such that $f \circ [~,~] = [~,~] \circ f^{\otimes 2}$ and $\Delta \circ f = f^{\otimes 2} \circ \Delta$.

If $\alpha = \beta = Id$ we recover Lie bialgebras and if $\alpha = \beta$, we recover the class defined in \cite{Yau:ClassicYangBaxter}, these   Hom-Lie bialgebras are denoted $(A,[~,~],\Delta,\alpha)$.
\end{definition}
In terms of elements, the compatibility condition \eqref{CompaCondHLieBialg} writes
\begin{equation} \label{compa_Lie_bialg}
\begin{split}
\Delta([x,y]) =& ad_{\alpha(x)}(\Delta(y)) - ad_{\alpha(y)}(\Delta(x)) \\
=& [\alpha(x),y_1] \otimes \beta(y_2) + \beta(y_1) \otimes [\alpha(x),y_2] \\
&- [\alpha(y),x_1] \otimes \beta(x_2) - \beta(x_1) \otimes [\alpha(y),x_2].
\end{split}
\end{equation}
\begin{remark}
If $\alpha = \beta$, the compatibility condition \eqref{compa_Lie_bialg} is equivalent to say that $\Delta$ is a $1$-cocycle with respect to  $\alpha^0$-adjoint cohomology of Hom-Lie algebras and for $\beta=Id$, it corresponds to $\alpha^{-1}$-adjoint cohomology, (see \cite{Sheng} and \cite{AEM}).
\end{remark}
The following Proposition generalizes slightly \cite[Theorem 3.5]{Yau:ClassicYangBaxter}.
\begin{proposition}
Let $(A,[~,~],\Delta,\alpha)$ be a Hom-Lie bialgebra and $\beta : A \to A$ a Hom-Lie bialgebra morphism commuting with $\alpha$. Then $(A,[~,~]_\beta = \beta \circ [~,~],\Delta_\beta = \Delta \circ \beta,\alpha \circ \beta)$ is a Hom-Lie bialgebra.
\end{proposition}

\begin{proof}
We already know that $(A,[~,~]_\beta,\beta \circ \alpha)$ is a Hom-Lie algebra and that $(A,\Delta_\beta,\alpha \circ \beta)$ is a Hom-Lie coalgebra. It remains to prove the compatibility condition \eqref{compa_Lie_bialg} for $\Delta_\beta$ and $[~,~]_\beta$, with the twisting map $\alpha \circ \beta = \beta \circ \alpha$. On one side, we have
\begin{equation*}
\Delta_\beta([x,y]) = \Delta \circ \beta^2 \circ [x,y] = \left(\beta^{\otimes 2}\right)^2 \circ \Delta([x,y]),
\end{equation*}
since $\Delta \circ \beta = \beta^{\otimes 2} \circ \Delta$. Using in addition $\beta \circ [~,~] = [~,~] \circ \beta^{\otimes 2}$ and the fact that $\alpha$ and $\beta$ commute, we have
\begin{equation*}
\begin{split}
ad_{\alpha \circ \beta(x)}(\Delta_\beta(y)) &= ad_{\alpha \circ \beta(x)}(\beta(y_1) \otimes \beta(y_2)) \\
&= [\alpha \circ \beta(x),\beta(y_1)]_\beta \otimes \alpha \circ \beta^2(y_2) + \alpha \circ \beta^2(y_1) \otimes [\alpha \circ \beta(x),\beta(y_2)]_\beta \\
&= \left(\beta^{\otimes 2}\right)^2([\alpha(x),y_1] \otimes \alpha(y_2) + \alpha(y_1) \otimes [\alpha(x),y_2]).
\end{split}
\end{equation*}
It follows that $\Delta_\beta([x,y]) = ad_{\alpha \circ \beta(x)}(\Delta_\beta(y)) - ad_{\alpha \circ \beta(y)}(\Delta_\beta(x))$ as wished.
\end{proof}

As for the Hom-bialgebra, the category of Hom-Lie bialgebra is not closed under weak Hom-Lie bialgebra morphisms.

Hom-Lie bialgebra can be dualized. We obtain the following proposition generalized the result stated in \cite{Yau:ClassicYangBaxter} for finite dimensional case and  using natural pairing.

\begin{proposition}
Let $(A,[~,~],\alpha,\Delta,\beta)$ be a Hom-Lie bialgebra. Then the finite dual $(A^\circ,[~,~]^\circ,\alpha^\circ,\Delta^\circ,\beta^\circ)$ is a Hom-Lie bialgebra as well.
\end{proposition}

\section{Hom-Poisson algebras}\label{Hom-Poisson}

The Hom-Poisson algebras were introduced in \cite{MakSil-def}, where they  emerged naturally in the study of 1-parameter formal deformations of commutative Hom-associative algebras. Then this structure was studied in   \cite{Yau:Poisson}. It is shown that they are closed under twisting by suitable self maps and under tensor products. Moreover it is shown that (de)polarization Hom-Poisson algebras are equivalent to admissible Hom-Poisson algebras, each of which has only one binary operation. We survey in this section the fundamental results and provide  examples. 

\begin{definition}
A \emph{Hom-Poisson} algebra is a tuple  $(A, \mu, \{~,~\}, \alpha)$ consisting  of
\begin{enumerate}[(1)]
\item a commutative Hom-associative algebra $(A, \mu, \alpha)$ and
\item a Hom-Lie algebra $(A, \{~,~\}, \alpha)$
\end{enumerate}
such that the Hom-Leibniz identity
\begin{equation}
\{~,~\} \circ (\mu \otimes \alpha) = \mu \circ (\alpha \otimes \{~,~\} + (\{~,~\} \otimes \alpha) \circ \tau_{23})
\end{equation}
is satisfied. 
\end{definition}
In a Hom-Poisson algebra $(A, \{~,~\}, \mu, \alpha)$, the operation $\{~,~\}$ is  called 
\emph{Hom-Poisson bracket}. In terms of elements $x, y, z \in A$, the Hom-Leibniz identity says
\begin{equation*}
\{xy,\alpha(z)\} = \alpha(x) \{y,z\} + \{x,z\} \alpha(y)
\end{equation*}
where as usual $\mu(x,y)$ is abbreviated to $xy$. By the antisymmetry of the Hom-Poisson bracket $\{~,~\}$, the Hom-Leibniz identity is equivalent to
\begin{equation*}
\{\alpha(x), yz\} = \{x,y\} \alpha(z) + \alpha(y) \{x,z\}.
\end{equation*}
We recover Poisson algebras when the twisting map is the  identity.

\begin{definition}\label{DefinitionHom-Poissonbialgebra}
A \emph{Hom-Poisson bialgebra} $(A,\mu,\eta,\Delta,\varepsilon,\alpha,\{~,~\})$ is a Hom-Poisson algebra $(A,\mu,\{~,~\},\alpha)$ which is also a Hom-bialgebra $(A,\mu,\eta,\Delta,\varepsilon,S,\alpha)$, the two structures being compatible in the sense that $\{~,~\}$ is a $\mu$-coderivation,
\begin{equation*}
\Delta \circ \{~,~\} = (\{~,~\} \otimes \mu + \mu \otimes \{~,~\}) \circ \Delta^{[2]}.
\end{equation*}
In term of elements, this compatibility condition writes
\begin{equation*}
\Delta(\{a,b\}) = \{\Delta(a),\Delta(b)\}
\end{equation*}
with the Hom-Poisson bracket on $A \otimes A$ defined by
\begin{equation*}
\{a_1 \otimes a_2,b_1 \otimes b_2\} := \{a_1,b_1\} \otimes a_2 b_2 + a_1 b_1 \otimes \{a_2,b_2\}.
\end{equation*}
\end{definition}
We have the same definition for Hom-Poisson Hopf algebras.

\begin{example}\label{example1HomPoisson}
Let $\{x_1,x_2,x_3\}$  be a basis of a $3$-dimensional vector space
$A$ over $\K$. The following multiplication $\mu$, skew-symmetric
bracket and linear map $\alpha$ on $A$ define a Hom-Poisson algebra
over $\K^3${\rm :}
\[
\begin{array}{ll}
 \begin{array}{lll}
 \mu(x_1,x_1) &=&  x_1, \ \\
 \mu(x_1,x_2) &=& \mu(x_2,x_1)=x_3,\\
 \end{array}
 & \quad
 \begin{array}{lll}
 \{ x_1,x_2 \}&=& a x_2+ b x_3, \ \\
 \{ x_1, x_3 \}&=& c x_2+ d x_3, \ \\
 \end{array}
\end{array}
\]

\[
\alpha (x_1)= \lambda_1 x_2+\lambda_2 x_3 , \quad
\alpha (x_2) =\lambda_3 x_2+\lambda_4 x_3  , \quad
\alpha (x_3)=\lambda_5 x_2+\lambda_6 x_3
\]
where $a,b,c,d,\lambda_1,\lambda_2,\lambda_3,\lambda_4,\lambda_5,\lambda_6$ are parameters
in $\K$.
\end{example}

\begin{theorem}[\cite{Yau:Poisson}]
Let $A = (A,\mu,\{~,~\})$ be a Poisson algebra, and $\alpha : A \to A$ be a linear map which is multiplicative for $\mu$ and $\{~,~\}$. Then $$A_\alpha = (A,\quad \mu_\alpha = \alpha \circ \mu,\quad \{~,~\}_\alpha = \alpha \circ \{~,~\},\quad \alpha)$$ is a Hom-Poisson algebra.
\end{theorem}

\begin{proof}
Using Theorem \eqref{twist-ass_Lie}, we already have that $(A, \mu_\alpha, \alpha)$ is a commutative Hom-associative algebra and that $(A, \{~,~\}_\alpha, \alpha)$ is a Hom-Lie algebra. It remains to check the Hom-Leibniz identity 
\begin{equation*}
\begin{split}
\{~,~\}_\alpha \circ (\mu_\alpha \otimes \alpha) &= \alpha \circ \{~,~\} \circ \alpha^{\otimes 2} \circ (\mu \otimes Id) \\
&= \alpha^2 \circ \{~,~\} \circ (\mu \otimes Id),
\intertext{since $A$ is a Poisson algebra}
&= \alpha^2 \circ \mu \circ (Id \otimes \{~,~\} + (\{~,~\} \otimes Id) \circ \tau_{2 3}) \\
&= \alpha \circ \mu \circ \alpha^{\otimes 2} \circ (Id \otimes \{~,~\} + (\{~,~\} \otimes Id) \circ\tau_{23}) \\
&= \mu_\alpha \circ (\alpha \otimes \{~,~\}_\alpha + (\{~,~\}_\alpha \otimes \alpha) \circ \tau_{23}),
\end{split}
\end{equation*}
so $A_\alpha$ is a Hom-Poisson algebra.
\end{proof}

\begin{example}
The Sklyanin Poisson algebra $q_4(\mathcal{E})$ (see \cite{ORP} for a more detailed definition and properties) is defined on $\C[x_0, x_1, x_2, x_3]$ by a parameter $k \in \C$ with the usual polynomial multiplication, and bracket given by $\{~,~\}$ with brackets between the coordinate functions defined as
\begin{equation*}
\begin{aligned}
\{x_i, x_{i+1}\} &= k^2 x_i x_{i+1} - x_{i+2} x_{i+3}, \\
\{x_i, x_{i+2}\} &= k (x^2_{i+3} - x^2_{i+1}),
\end{aligned}
\qquad i = 1, 2, 3, 4 \pmod 4.
\end{equation*}

We again search a morphism $\alpha : q_4(\mathcal{E}) \to q_4(\mathcal{E})$ written as $\alpha = (\alpha_{ij})$ with respect to the  basis $(x_0,x_1,x_2,x_3)$, by solving the coefficients $\alpha_{ij}$ in the equations $\alpha([x_i,x_j]) = [\alpha(x_i),\alpha(x_j)]$ with respect to the basis. For simplicity, we take $\alpha$ diagonal, $\alpha_{ij} = 0$ if $i \neq j$. \\
We obtain $q_4(\mathcal{E})_\alpha$, Hom-Poisson versions of $q_4(\mathcal{E})$, with $\alpha$ given by
\begin{enumerate}
\item $\alpha(x_0) = - \lambda x_0, \ \alpha(x_1) = i \lambda x_1, \ \alpha(x_2) = \lambda x_2, \ \alpha(x_3) = -i \lambda x_3$,
\item $\alpha(x_0) = - \lambda x_0, \ \alpha(x_1) = -i \lambda x_1, \ \alpha(x_2) = \lambda x_2, \ \alpha(x_3) = i \lambda x_3$,
\item $\alpha(x_0) = \lambda x_0, \ \alpha(x_1) = - \lambda x_1, \ \alpha(x_2) = \lambda x_2, \ \alpha(x_3) = - \lambda x_3$,
\item $\alpha = \lambda id$,
\end{enumerate}
with  $\lambda\in\K$.

For example, $q_4(\mathcal{E})$ carries a structure of Hom-Poisson algebra, for any $\lambda\in \C$, with  the following bracket
\begin{equation*}
\begin{aligned}
\{x_i, x_{i+1}\} &=-\lambda^2 ( k^2 x_i x_{i+1} - x_{i+2} x_{i+3}), \\
\{x_i, x_{i+2}\} &= \lambda^2  k (x^2_{i+3} - x^2_{i+1}),
\end{aligned}
\qquad i = 1, 2, 3, 4 \pmod 4,
\end{equation*}
and linear map
\[
\alpha(x_0) = \lambda x_0, \ \alpha(x_1) = - \lambda x_1, \ \alpha(x_2) = \lambda x_2, \ \alpha(x_3) = - \lambda x_3.
\]
\end{example}

\subsection{Constructing Hom-Poisson algebras from Hom-Lie algebras}

Suppose that $(A, [~,~], \alpha)$ is a finite-dimensional Hom-Lie algebra and  $\{e_i\}_{1 \leqslant i \leqslant n}$ be a basis of $A$. Set $C_{ij}^k$ be the structure constants of the bracket with respect to the basis, that is $[e_i,e_j] = \sum_{k=1}^n C_{ij}^k e_k$ and $\alpha_i^s$ be the coefficients of the morphism $\alpha$,  that is $\alpha(e_i) = \sum_{s=1}^n \alpha_i^s e_s$. \\
The skew-symmetry  of the bracket and the Hom-Jacobi condition can be written with the structure constants as 
\begin{gather*}
C_{ji}^k = - C_{ij}^k \qquad \textrm{antisymmetry}, \\
\sum_{1 \leqslant p,q \leqslant n} (\alpha_i^p C_{jk}^q + \alpha_j^p C_{ki}^q + \alpha_k^p C_{ij}^q)C_{pq}^s = 0 \qquad \textrm{Hom-Jacobi identity}.
\end{gather*}

To construct a Hom-Poisson algebra from a Hom-Lie algebra, we should define a commutative multiplication $\cdot$ which is Hom-associative and a bracket $\{~,~\}$ satisfying the Hom-Leibniz identity. Define the bracket $\{~,~\}$ as being equal to the bracket $[~,~]$ on the basis, and extended by the Hom-Leibniz identity. \\
Set $M_{ij}^k$ be the structure constants for the multiplication,  that is  $e_i \cdot e_j = \sum_{k=1}^n M_{ij}^k e_k$. By commutativity, $M_{ji}^k = M_{ij}^k$. The Hom-Leibniz identity writes
\begin{align*}
0 &= \{e_i \cdot e_j, \alpha(e_k)\} - \alpha(e_i) \cdot \{e_j,e_k\} - \{e_i,e_k\} \cdot \alpha(e_j) \\
\Leftrightarrow 0 &= \sum_{s=1}^n \underbrace{\big( M_{ij}^p \alpha_k^q C_{pq}^s - (\alpha_i^p C_{jk}^q + C_{ik}^p \alpha_j^q) M_{pq}^s \big)}_{S_{ijks}} e_s \\
\Leftrightarrow 0 &= S_{ijks},
\end{align*}
giving a linear system in the $M_{ij}^l$ of $n^4$ equations in $n^3$ unknowns\footnote{actually $\frac{n^2(n+1)}{2}$ unknowns using the commutativity}.

The Hom-associativity writes
\begin{align*}
0 &= (e_i \cdot e_j) \cdot \alpha(e_k) - \alpha(e_i) \cdot (e_j \cdot e_k) \\
\Leftrightarrow 0 &= \sum_{s=1}^n \underbrace{\big( (M_{ij}^p \alpha_k^q - \alpha_i^p M_{jk}^q) M_{pq}^s \big)}_{R_{ijks}} e_s \\
\Leftrightarrow 0 &= R_{ijks},
\end{align*}
giving a non linear system in the $M_{ij}^l$ of $n^4$ equations in $n^3$ unknowns.

Solving first the equations of Hom-Leibniz and then checking if the solutions satisfy the Hom-associativity equations, we obtain example of Hom-Poisson algebras.

\begin{example}
We consider the $3$-dimensional Hom-Lie algebra with basis $\{e_1,e_2,e_3\}$, brackets given by
\begin{align*}
[e_1,e_2] &= C_{12}^2 e_2 + C_{12}^3 e_3 \\
[e_1,e_3] &= C_{13}^2 e_2 + C_{13}^3 e_3 \\
[e_2,e_3] &= 0,
\end{align*}
and morphism $\alpha = \left(\begin{smallmatrix} 0 & 0 & 0 \\ 0 & b & 0 \\ 0 & 0 & b \end{smallmatrix}\right)$ in the basis $\{e_1,e_2,e_3\}$.
The only multiplications giving a Hom-Poisson algebra are of the form
\begin{align*}
e_1 \cdot e_2 &= \lambda e_2 \\
e_1 \cdot e_3 &= \lambda e_3 \\
e_2 \cdot e_3 &= 0 \\
e_i \cdot e_i &= 0 \quad \text{for}\ i=1,2,3.
\end{align*}
\end{example}

\begin{example}
Other examples of Hom-Poisson algebras of dimension $3$ with basis $\{e_1,e_2,e_3\}$ are given by twisting the following Poisson algebra:
\begin{equation*}
e_1 \cdot e_1 = e_2 \qquad \{e_1,e_3\} = a e_2 + b e_3,
\end{equation*}
with all other multiplication and brackets equal to zero.

The morphism $\alpha$ is computed to be multiplicative for the multiplication and the bracket.
\begin{center}
\begin{tabular}{c|c}
With $a \neq 0,\ b \neq 0$ & With $a \neq 0, b=0$ \\
$\alpha(e_1)= e_1 + \alpha_{12} e_2 + \alpha_{13} e_3$ & $\alpha(e_1) = c e_1 + \alpha_{12} e_2 + \alpha_{13} e_3$ \\
$\alpha(e_2) = e_2$ & $\alpha(e_2) = c^2 e_2$ \\
$\alpha(e_3) = \alpha_{32} e_2 + \frac{b}{a} \alpha_{32} e_3$ & $\alpha(e_3) = \alpha_{31} e_1 + \alpha_{32} e_2 + \alpha_{33} e_3$ \\
& where $c$ is a solution (if it exist) \\
& of $X^2 - \alpha_{33} X + \alpha_{13} \alpha_{31} = 0$.
\end{tabular}
\end{center}
\end{example}

\subsection{Flexibles structures}

We recall here some results on flexible structures described in \cite{MakSil-struct} and provide a connection to Hom-Poisson algebras.

\begin{definition}
A Hom-algebra $A = (A, \mu, \alpha)$ is called flexible if for any $x, y \in A$
\begin{equation} \label{flex_eq}
\mu(\mu(x, y), \alpha(x)) = \mu(\alpha(x), \mu(y, x))).
\end{equation}
\end{definition}

\begin{remark}
Using the Hom-associator $\as_{\mu,\alpha}$, the condition \eqref{flex_eq} may be written as
\begin{equation*}
\as_{\mu,\alpha}(x, y, x) = 0.
\end{equation*}
\end{remark}

\begin{lemma} \label{flex_lemma} Let $A = (A, \mu, \alpha)$ be a Hom-algebra. The following assertions are equivalent
\begin{enumerate}[(1)]
\item $A$ is flexible.
\item For any $x, y \in A,\quad \as_{\mu,\alpha}(x, y, x) = 0$.
\item For any $x, y, z \in A,\quad \as_{\mu,\alpha}(x, y, z) = -\as_{\mu,\alpha}(z, y, x)$.
\end{enumerate}
\end{lemma}

\begin{proof}
The equivalence of the first two assertions follows from the definition. The  assertion  $\as_{\mu,\alpha}(x - z, y, x - z) = 0$ holds by definition and it is equivalent to $\as_{\mu,\alpha}(x, y, z) + \as_{\mu,\alpha}(z, y, x) = 0$ by linearity.
\end{proof}

\begin{corollary}
Any Hom-associative algebra is flexible.
\end{corollary}

Let $A = (A, \mu, \alpha)$ be a Hom-algebra, where $\mu$ is the multiplication and $\alpha$ a homomorphism. We define two new multiplications using  $\mu$ :
\begin{equation*}
\forall\ x,y \in A \quad x \bullet y = \mu(x, y)+\mu(y, x),  \quad \{x,y\} = \mu(x,y) - \mu(y,x).
\end{equation*}
We set $A^+ = (A, \bullet, \alpha)$ and $A^- = (A, \{~,~\},\alpha)$.

\begin{proposition} \label{flex_Leibniz_Proposition} A Hom-algebra $A = (A, \mu, \alpha)$ is flexible if and only if
\begin{equation} \label{flex_Leibniz}
\{\alpha(x), y \bullet z\} = \{x, y\} \bullet \alpha(z) + \alpha(y) \bullet \{x, z\}.
\end{equation}
\end{proposition}

\begin{proof}
Let A be a flexible Hom-algebra. By Lemma \eqref{flex_lemma}, this is equivalent to $\as_{\mu,\alpha}(x, y, z) + \as_{\mu,\alpha}(z, y, x) = 0$ for any $x, y, z \in A$.  This implies
\begin{eqnarray} \label{flex_Leibniz_equiv}
\as_{\mu,\alpha}(x, y, z) + \as_{\mu,\alpha}(z, y, x) + \as_{\mu,\alpha}(x, z, y) +\\ \nonumber  \as_{\mu,\alpha}(y, z, x) - \as_{\mu,\alpha}(y, x, z) - \as_{\mu,\alpha}(z, x, y) = 0
\end{eqnarray}
By expansion, the previous relation is equivalent to $\{\alpha(x), y \bullet z\} = \{x, y\} \bullet \alpha(z) + \alpha(y) \bullet \{x, z\}$. Conversely, assume we have the condition \eqref{flex_Leibniz} in Proposition. By setting $x = z$ in \eqref{flex_Leibniz_equiv}, one
gets $\as_{\mu,\alpha}(x, y, x) = 0$. Therefore $A$ is flexible.
\end{proof}

Hence, we obtain the following connection to Hom-Poisson algebras. 
\begin{proposition}
Let  $A = (A, \mu, \alpha)$ be a flexible Hom-algebra which is Hom-associative. Then $(A, \bullet, \{~,~\}, \alpha)$, where $\bullet$ and $\{~,~\}$ are the operations defining $A^+$ and $A^-$ respectively, is a Hom-Poisson algebra.
\end{proposition}

\subsection{\texorpdfstring{$1$-operation Poisson algebras}{1-operation Poisson algebras}}

Algebras with one operation were introduced by  Loday and studied by Markl and  Remm in \cite{MR}. The twisted version was studied in  \cite{Yau:Poisson} where they are called admissible Hom-Poisson algebras.

\begin{definition}
A \emph{$1$-operation Hom-Poisson algebra} is a Hom-algebra $(A, \cdot , \alpha)$ satisfying, for any $x, y, z \in A$,
\begin{equation} \label{1op_hom-Poisson}
3 \as_{\cdot,\alpha}(x,y,z) = (x \cdot z) \cdot \alpha(y) + (y \cdot z) \cdot \alpha(x) - (y \cdot x) \cdot \alpha(z) - (z \cdot x) \cdot \alpha(y).
\end{equation}
\end{definition}
If $\alpha$ is the identity map, $A$ is called a $1$-operation Poisson algebra.

We consider a Hom-algebra $(A, \cdot , \alpha)$. We define two operations $\bullet : A \otimes A \to A$ and $\{~,~\} : A \otimes A \to A$ by
\begin{equation}
\forall\ x,y \in A, \qquad \qquad x \bullet y = x \cdot y + y \cdot x, \qquad \{x,y\} = x \cdot y - y \cdot x.
\end{equation}

\begin{theorem}
$(A, \bullet , \{~,~\}, \alpha)$ is a Hom-Poisson algebra if and only if $(A, \cdot , \alpha)$ is a $1$-operation Hom-Poisson algebra.
\end{theorem}

\begin{proof}
Suppose that $(A, \bullet , \{~,~\}, \alpha)$ is a Hom-Poisson algebra.
Since
\begin{equation}
\forall\ x,y \in A,\ x \cdot y = \frac{1}{2}(\{x,y\} + x \bullet y),
\end{equation}
we have, by expansion,
\begin{equation*}
\as_{\cdot,\alpha}(x,y,z) = (x \cdot y) \cdot \alpha(z) - \alpha(x) \cdot (y \cdot z) = \frac{1}{4} \{\alpha(y),\{z,x\}\},
\end{equation*}
and, on the other hand, using that the multiplication $\bullet$ is Hom-associative and commutative, and that $\{~,~\}$ is a Hom-Lie bracket,
\begin{equation*}
(x \cdot z) \cdot \alpha(y) + (y \cdot z) \cdot \alpha(x) - (y \cdot x) \cdot \alpha(z) - (z \cdot x) \cdot \alpha(y) = \frac{3}{4} \{\alpha(y),\{z,x\}\}.
\end{equation*}
We thus have the equation \eqref{1op_hom-Poisson}.

Suppose now that the equation \eqref{1op_hom-Poisson} is verified. We have to show that $\bullet$ is Hom-associative and that $\{~,~\}$ is a Hom-Lie bracket.

Using the relation \eqref{1op_hom-Poisson}, we obtain the identities
\begin{gather}
\forall\ x,y,z \in A \quad \as_{\cdot,\alpha}(x,y,z) + \as_{\cdot,\alpha}(y,z,x) + \as_{\cdot,\alpha}(z,x,y) = 0  \label{as_cyclic} \\
\forall\ x,y,z \in A \quad \as_{\cdot,\alpha}(x,y,z) + \as_{\cdot,\alpha}(z,y,x) = 0. \label{as_sum}
\end{gather}

This last identity \eqref{as_sum} shows that $(A, \cdot , \alpha)$ is a Hom-flexible algebra using  Lemma \ref{flex_lemma}.

We now obtain
\begin{equation*}
\begin{split}
& \as_{\bullet,\alpha}(x,y,z) = (x \bullet y) \bullet \alpha(z) - \alpha(x) \bullet (y \bullet z) \\
=& \as_{\cdot,\alpha}(y,z,x) + \as_{\cdot,\alpha}(x,z,y) - (\as_{\cdot,\alpha}(z,y,x) + \as_{\cdot,\alpha}(x,y,z)) \\
\stackrel{\eqref{as_sum}}{=}& 0.
\end{split}
\end{equation*}
So the product $\bullet$ is Hom-associative and commutative by definition. Moreover,
\begin{equation*}
\begin{split}
&J_{\{~,~\},\alpha}(x,y,z) = \{\alpha(x),\{y,z\}\} + \{\alpha(y),\{z,x\}\} + \{\alpha(z),\{x,y\}\} \\
=& - (\as_{\cdot,\alpha}(x,y,z) + \as_{\cdot,\alpha}(y,z,x) + \as_{\cdot,\alpha}(z,x,y)) +\\
\ &  \as_{\cdot,\alpha}(x,z,y) + \as_{\cdot,\alpha}(y,x,z) + \as_{\cdot,\alpha}(z,y,x) \\
\stackrel{\eqref{as_cyclic}}{=}& 0,
\end{split}
\end{equation*}
so $\{~,~\}$ is a Hom-Lie bracket. Since $A$ is flexible,  Proposition \ref{flex_Leibniz_Proposition} leads to  the compatibility between $\bullet$ and $\{~,~\}$,
\begin{equation*}
\{\alpha(x),y \bullet z\} = \{x,y\} \bullet \alpha(z) + \alpha(y) \bullet \{x,z\}.
\end{equation*}
So $(A, \bullet , \{~,~\} , \alpha)$ is a Hom-Poisson algebra.
\end{proof}

\begin{proposition}
Let $(A, \cdot)$ be a $1$-operation Poisson algebra, and $\alpha : A \to A$ be a linear map multiplicative for the multiplication $\cdot$, i.e. $\alpha \circ \cdot = \cdot \circ \alpha^{\otimes 2}$, then $A_\alpha = (A, \cdot_\alpha = \alpha \circ \cdot, \alpha)$ is a $1$-operation Hom-Poisson algebra.
\end{proposition}

\begin{proof}
We have
\begin{equation*}
\begin{split}
3 \as_{\cdot_\alpha,\alpha}(x,y,z) &= (x \cdot_\alpha y) \cdot_\alpha \alpha(z) - \alpha(x) \cdot_\alpha (y \cdot_\alpha z) \\
&= \alpha(\alpha(x \cdot y) \cdot \alpha(z)) - \alpha(\alpha(x) \cdot \alpha(y \cdot z)) = \alpha^2((x \cdot y) \cdot z - x \cdot (y \cdot z)),
\intertext{and since $\cdot$ verifies the $1$-operation equation,}
3 \as_{\cdot_\alpha,\alpha}(x,y,z) &= \alpha^2((x \cdot z) \cdot y + (y \cdot z) \cdot x - (y \cdot x) \cdot z - (z \cdot x) \cdot y) \\
&= \alpha(\alpha(x \cdot z) \cdot \alpha(y)) + \alpha(\alpha(y \cdot z) \cdot \alpha(x)) \\ 
& \ \ - \alpha(\alpha(y \cdot x) \cdot \alpha(z)) - \alpha(\alpha(z \cdot x) \cdot \alpha(y)) \\
&= (x \cdot_\alpha z) \cdot_\alpha(y) + (y \cdot_\alpha z) \cdot_\alpha \alpha(x) - (y \cdot_\alpha x) \cdot_\alpha \alpha(z) - (z \cdot_\alpha x) \cdot_\alpha \alpha(y).
\end{split}
\end{equation*}
\end{proof}

\section{Hom-coPoisson structures} \label{Hom-coPoisson}

\begin{definition}
A \emph{Hom-coPoisson algebra} consists of a cocommutative coassociative Hom-coalgebra $(A,\Delta,\varepsilon,\alpha)$ equipped with a skew-symmetric linear map $\delta : A \to A \otimes A$, the Hom-coPoisson cobracket, satisfying the following conditions
\begin{enumerate}[(i)]
\item (Hom-coJacobi identity)
\begin{equation} \label{Hom-coJacobi}
(Id+\sigma+\sigma^2) \circ (\alpha \otimes \delta) \circ \delta = 0,
\end{equation}
\item (Hom-coLeibniz rule)
\begin{equation} \label{Hom-coLeibniz}
(\Delta \otimes \alpha) \circ \delta = (\alpha \otimes \delta) \circ \Delta + \tau_{23} \circ (\delta \otimes \alpha) \circ \Delta.
\end{equation}
\end{enumerate}
It is denoted by a tuple $(A,\Delta,\varepsilon,\alpha,\delta)$.
\end{definition}

\begin{proposition}
Let $(A,\Delta,\varepsilon,\alpha,\delta)$ be a Hom-coPoisson algebra and $\beta: A \to A$ be a Hom-coPoisson algebra morphism. Then $(A,\Delta_\beta = \Delta \circ \beta,\varepsilon,\alpha \circ \beta,\delta_\beta = \delta \circ \beta)$ is a Hom-coPoisson algebra.
\end{proposition}
\begin{proof}
Theorem \ref{thmConstrHomCoalg} insures that $(A,\Delta_\beta,\varepsilon,\alpha \circ \beta)$ is a coassociative Hom-coalgebra and Theorem \ref{thmConstrHomLieCoalg} that $(A,\delta_\beta,\alpha \circ \beta)$ is a Hom-Lie coalgebra. It remains to show the compatibility  condition \ref{Hom-coLeibniz}. On the left hand side, we have
\begin{equation*}
(\Delta_\beta \otimes \alpha \circ \beta) \circ \delta_\beta = (\Delta \circ \beta \otimes \alpha \circ \beta) \circ \delta \circ \beta = (\Delta \otimes \alpha) \circ \beta^2,
\end{equation*}
and the right hand side gives
\begin{align*}
(\alpha &\circ \beta \otimes \delta_\beta) \circ \Delta_\beta + \tau_{23} \circ (\delta_\beta \otimes \alpha \circ \beta) \circ \Delta_\beta \\
&= (\alpha \circ \beta \otimes \delta \circ \beta) \circ \Delta \circ \beta + \tau_{23} \circ (\delta \circ \beta \otimes \alpha \circ \beta) \circ \Delta \circ \beta \\
&= \left[ (\alpha \otimes \delta) \circ \Delta + \tau_{23} \circ (\delta \otimes \alpha) \circ \Delta \right] \circ \beta^2,
\end{align*}
which ends the proof.
\end{proof}

We may state the following Corollaries. Starting from a classical coPoisson algebra, we may construct Hom-coPoisson algebras using coPoisson algebra morphisms. On the other hand a Hom-coPoisson algebra gives rise to infinitely many Hom-coPoisson algebras.

\begin{corollary} \hfill
\begin{enumerate}
\item Let $(A,\Delta,\varepsilon,\delta)$ be a coPoisson algebra and $\beta : A \to A$ be a coPoisson algebra morphism.
Then $(A,\Delta_\beta = \Delta \circ \beta,\varepsilon,\beta,\delta_\beta = \delta \circ \beta)$ is a Hom-coPoisson algebra.
\item Let $(A,\Delta,\varepsilon,\alpha,\delta)$  be a Hom-coPoisson algebra.
Then for any non negative integer $n$, we have  $(A,\Delta\circ\alpha^n,\varepsilon,\alpha^{n+1},\delta\circ\alpha^n)$ is a Hom-coPoisson algebra.
\end{enumerate}
\end{corollary}

\begin{definition}
A \emph{Hom-coPoisson bialgebra} $(A,\mu,\eta,\Delta,\varepsilon,\alpha,\delta)$ is a Hom-coPoisson algebra $(A,\Delta,\varepsilon,\alpha,\delta)$ which is also a Hom-bialgebra $(A,\mu,\eta,\Delta,\varepsilon,\alpha)$, the two structures being compatible in the sense that $\delta$ is a $\Delta$-derivation,
\begin{equation*} \label{Delta-der}
\delta \circ \mu =(\mu\otimes\mu)\circ\tau_{23}\circ(\delta \otimes \Delta + \Delta \otimes \delta).
\end{equation*}
A \emph{Hom-coPoisson Hopf algebra} $(A,\mu,\eta,\Delta,\varepsilon,S,\alpha,\delta)$ is a Hom-coPoisson bialgebra $(A,\mu,\eta,\Delta,\varepsilon,\alpha,\delta)$ with an antipode $S$, such that the tuple $(A,\mu,\eta,\Delta,\varepsilon,S,\alpha)$ is a Hom-Hopf algebra.
\end{definition}

We extend  the connection  between Lie bialgebras and coPoisson-Hopf algebras presented in \cite{ChariPressley} to the Hom setting.
\begin{proposition}
Let $(\mathfrak{g},[~,~],\alpha)$ be a Hom-Lie algebra. If its universal enveloping algebra $U_{HLie}(\mathfrak{g})$ has a Hom-coPoisson structure $\delta$, making it a Hom-coPoisson bialgebra, then $\delta(\mathfrak{g}) \subset \mathfrak{g} \otimes \mathfrak{g}$, and $\delta|_\mathfrak{g}$ equips $(\mathfrak{g},[~,~],\alpha,\delta,Id)$ with a structure of Hom-Lie bialgebra. Conversely, for a Hom-Lie bialgebra $(\mathfrak{g},[~,~],\delta,\alpha)$, the cobracket $\delta : \mathfrak{g} \to \mathfrak{g} \otimes \mathfrak{g}$ extends uniquely to a Hom-coPoisson cobracket on $U_{HLie}(\mathfrak{g})$, which makes $U_{HLie}(\mathfrak{g})$ a Hom-coPoisson bialgebra.
\end{proposition}

\begin{proof}
Let $\delta : U_{HLie}(\mathfrak{g}) \to U_{HLie}(\mathfrak{g}) \otimes U_{HLie}(\mathfrak{g})$ be a Hom-coPoisson cobracket on $U_{HLie}(\mathfrak{g})$. To show that $\delta(\mathfrak{g}) \subset \mathfrak{g} \otimes \mathfrak{g}$, let $x \in \mathfrak{g}$, and write $\delta(x) = \sum_{(x)} x^{(1)} \otimes x^{(2)}$ where $x^{(1)},x^{(2)} \in U_{HLie}(\mathfrak{g})$. We may assume that the $x^{(2)}$ are linearly independent. By the Hom-coLeibniz condition \eqref{Hom-coLeibniz}, we have
\begin{equation*}
\begin{split}
\sum_{(x)} \Delta(x^{(1)}) \otimes \alpha(x^{(2)}) =& \alpha(1) \otimes \delta(\alpha(x)) + \alpha(\alpha(x)) \otimes \delta(1) \\[-2ex]
& + \tau_{23} \circ \left( \delta(1) \otimes \alpha(\alpha(x)) + \delta(\alpha(x)) \otimes \alpha(1) \right)
\end{split}
\end{equation*}
since $x \in \mathfrak{g}$ and $\Delta(x) = 1 \otimes \alpha(x) + \alpha(x) \otimes 1$. Moreover, $\alpha$ is a morphism so $\alpha(1) = 1$ and $\delta$ is a $\Delta$-derivation so $\delta(1) = 0$, hence
\begin{align*}
\sum_{(x)} \Delta(x^{(1)}) \otimes \alpha(x^{(2)}) &= 1 \otimes \delta(\alpha(x)) + \tau_{23} \circ \left( \delta(\alpha(x)) \otimes 1 \right) \\
&= \sum_{(x)} \left( 1 \otimes \alpha(x^{(1)}) + \alpha(x^{(1)}) \otimes 1 \right) \otimes \alpha(x^{(2)})
\end{align*}
using the multiplicativity $\delta \circ \alpha = \alpha^{\otimes 2} \circ \delta$ of the Hom-coPoisson morphism $\alpha$. It follows that the $x^{(1)}$ are Hom-primitive elements ($\Delta (x)=1\otimes \alpha (x)+\alpha (x)\otimes 1$) of $U_{HLie}(\mathfrak{g})$, hence $\delta(\mathfrak{g}) \subset \mathfrak{g} \otimes U_{HLie}(\mathfrak{g})$. Since $\delta$ is skew-symmetric,
\[
\delta(\mathfrak{g}) \subset \left( \mathfrak{g} \otimes U_{HLie}(\mathfrak{g}) \right) \cap  \left( U_{HLie}(\mathfrak{g}) \otimes \mathfrak{g} \right) = \mathfrak{g} \otimes \mathfrak{g}.
\]
To prove the compatibility condition \eqref{compa_Lie_bialg} for $\delta|_\mathfrak{g}$ and the twisting maps $\alpha$ and $Id$, let $x,y \in \mathfrak{g}$ and compute
\begin{align*}
\delta([x,y]) =& \delta(xy-yx) \\
=& \delta(x)\Delta(y) + \Delta(x)\delta(y) - \delta(y)\Delta(x) - \Delta(y)\delta(x) \\
=& [\Delta(x),\delta(y)] - [\Delta(y),\delta(x)] \\
=& [\alpha(x),y^{(1)}] \otimes y^{(2)} + y^{(1)} \otimes [\alpha(x),y^{(2)}] \\
&- [\alpha(y),x^{(1)}] \otimes x^{(2)} - x^{(1)} \otimes [\alpha(y),x^{(2)}] \\
=& ad_{\alpha(x)}(\Delta(y)) - ad_{\alpha(y)}(\Delta(y)).
\end{align*}
Conversely, $\delta : \mathfrak{g} \to \mathfrak{g} \otimes \mathfrak{g}$ uniquely extends in $\overline{\delta} : U_{HLie}(\mathfrak{g}) \to U_{HLie}(\mathfrak{g}) \otimes U_{HLie}(\mathfrak{g})$ such that $\overline{\delta}|_\mathfrak{g} = \delta$, with the formula
\[
\overline{\delta}(xy) = \overline{\delta}(x)\Delta(y)+\Delta(x)\overline{\delta}(y).
\]
This gives $U_{HLie}(\mathfrak{g})$ a structure of Hom-coPoisson bialgebra.
\end{proof}

\subsection{Duality}

In this section, we  extend to Hom-algebras the result stated in \cite{OhPark}, that the Hopf dual of a coPoisson Hopf algebra is a Poisson-Hopf algebra.

\begin{definition}
An algebra $A$ over $\K$ is said to be an almost normalizing extension over $\K$ if $A$ is a finitely generated $\K$-algebra with generators $x_1,\dotsc,x_n$ satisfying the condition
\begin{equation}
x_i x_j - x_j x_i \in \sum_{l=1}^n \K x_l + \K
\end{equation}
for all $i,j$.
\end{definition}

\begin{lemma}
Let $A$ be an almost normalizing extension of $\K$ with generators $x_1,\dotsc,x_n$. Then $A$ is spanned by all standard monomials
\begin{equation*}
x_1^{r_1} x_2^{r_2} \dotsm x_n^{r_n}, \qquad r_i \in \N
\end{equation*}
together with the unity 1.
\end{lemma}

\begin{proof}
This follows immediately from induction on the degree of monomials.
\end{proof}

Recall that the bialgebra (resp. Hopf) dual $A^\circ$ of a Hom-bialgebra (resp.Hom-Hopf algebra) $A$ consists of
\begin{equation*}
A^\circ = \{f \in A^*, f(I)=0\ \text{for some cofinite ideal $I$ of $A$}\},
\end{equation*}
where $A^*$ is the dual vector space of $A$.

\begin{theorem} \label{thmHopfDual}
Let $A$ be a Hom-coPoisson bialgebra (resp. Hopf algebra) with Poisson co-bracket $\delta$ and twisting map $\alpha$. If $A$ is an almost normalizing extension over $\K$, then the bialgebra (resp. Hopf) dual $A^\circ$ is a Hom-Poisson bialgebra (resp. Hopf algebra) with twisting map $\alpha^\circ$ and bracket
\begin{equation} \label{dual-bracket}
\{f,g\}(x) = \langle \delta(x),f \otimes g \rangle, \qquad x \in A
\end{equation}
for any $f,g \in A^\circ$, where $\langle \cdot,\cdot \rangle$ is the natural pairing between the vector space $A \otimes A$ and its dual vector space.
\end{theorem}

\begin{proof}
The proof is almost the same as in \cite{OhPark}. We do not reproduce here the first step showing that the bracket \eqref{dual-bracket} is well-defined, it uses the fact that $A$ is an almost normalizing extension over $\K$.

The skew-symmetry follows from $\tau_{12} \circ \delta = -\delta$, we have
\begin{align*}
\{g,f\}(x) &= \langle \delta(x),g \otimes f \rangle = \langle \tau_{12} \circ \delta, f \otimes g \rangle \\
&= - \langle \delta(x), f \otimes g \rangle = - \{f,g\}(x),
\end{align*}
for all $x \in A$.

The equation \eqref{dual-bracket} satisfies the Hom-Leibniz rule: since
\begin{equation}
\{fg,\alpha^\circ(h)\}(x) = \langle (\Delta \otimes \alpha) \circ \delta(x), f \otimes g \otimes h \rangle
\end{equation}
and
\begin{align*}
& (\alpha^\circ(f) \{g,h\} + \{f,h\} \alpha^\circ(g))(x) \\
&= \langle (\alpha \otimes \delta) \circ \Delta(x),f \otimes g \otimes h \rangle + \langle \tau_{23} \circ (\delta \otimes \alpha) \circ \Delta(x),f \otimes g \otimes h \rangle
\end{align*}
for $x \in A$ and $f,g,h \in A^\circ$, it is enough to show that
\begin{equation}
(\Delta \otimes \alpha) \circ \delta = (\alpha \otimes \delta) \circ \Delta + \tau_{23} \circ (\delta \otimes \alpha) \circ \Delta,
\end{equation}
but this is just the Hom-coLeibniz rule for $\delta$.

The equation \eqref{dual-bracket} satisfies the Hom-Jacobi identity: we have
\begin{align*}
\{\alpha^\circ(f),\{g,h\}\}(x) &= \langle (\alpha \otimes \delta) \circ \delta(x),f \otimes g \otimes h \rangle, \\
\{\alpha^\circ(g),\{h,f\}\}(x) &= \langle \sigma \circ (\alpha \otimes \delta) \circ \delta(x),f \otimes g \otimes h \rangle, \\
\{\alpha^\circ(h),\{f,g\}\}(x) &= \langle \sigma^2 \circ (\alpha \otimes \delta) \circ \delta(x),f \otimes g \otimes h \rangle,
\end{align*}
for $x \in A$ and $f,g,h \in A^\circ$. Hence \eqref{dual-bracket} satisfies the Hom-Jacobi identity if and only if $\delta$ satifies
\begin{equation*}
(Id+\sigma+\sigma^2) \circ (\alpha \otimes \delta) \circ \delta = 0,
\end{equation*}
which is just the Hom-coJacobi identity of $\delta$.

The bracket defined by \eqref{dual-bracket} satisfies the compatibility condition with the comultiplication of the Hom-bialgebra (resp. Hopf algebra), it is a $\mu$-coderivation: since $\delta$ is a $\Delta$-derivation, we have for $f,g \in A^\circ$
\begin{align*}
\Delta(\{f,g\})(x \otimes y) &= \{f,g\}(x y) = \langle \delta(x y),f \otimes g \rangle \\
&= \langle \delta(x) \Delta(y),f \otimes g \rangle + \langle \Delta(x) \delta(y),f \otimes g \rangle \\
&= \langle \delta(x),f_1 \otimes g_1 \rangle \langle \Delta(y),f_2 \otimes g_2 \rangle \\
&\quad + \langle \Delta(x),f_1 \otimes g_2 \rangle \langle \delta(y),f_2 \otimes g_2 \rangle \\
&= \{f_1,g_1\}(x) (f_2 g_2)(y) + (f_1 g_1)(x) \{f_2,g_2\}(y) \\
&= \{\Delta(f),\Delta(g)\}(x \otimes y).
\end{align*}

Finally, the bracket defined by \eqref{dual-bracket} equips $A^\circ$ with the structure of a Hom-Poisson bialgebra (resp. Hopf algebra), the twisting map being $\alpha^\circ$.
\end{proof}

Let $(\mathfrak{g},[~,~],\delta,\alpha)$ be a Hom-Lie bialgebra, $U_{HLie}(\mathfrak{g})$ the universal enveloping Hom-bialgebra of $\mathfrak{g}$ with comultiplication $\Delta$. The cobracket $\delta : \mathfrak{g} \to \mathfrak{g} \otimes \mathfrak{g}$ is extended uniquely to a $\Delta$-derivation $\overline{\delta} : U_{HLie}(\mathfrak{g}) \to U_{HLie}(\mathfrak{g}) \otimes U_{HLie}(\mathfrak{g})$ such that $\overline{\delta}|_{\mathfrak{g}} = \delta$ and $\overline{\delta}(xy) = \overline{\delta}(x) \Delta(y) + \Delta(x) \overline{\delta}(y)$. Then $U_{HLie}(\mathfrak{g})$ is a Hom-coPoisson bialgebra with cobracket $\overline{\delta}$.

\begin{corollary}
Let $(\mathfrak{g},[~,~],\delta,\alpha)$ be a finite dimensional Hom-Lie bialgebra. Then the  dual $U_{HLie}(\mathfrak{g})^\circ$ of the universal enveloping Hom-bialgebra $U_{HLie}(\mathfrak{g})$ is a Hom-Poisson bialgebra with Poisson bracket
\begin{equation*}
\{f,g\}(x) = \langle \overline{\delta}(x), f \otimes g \rangle, \qquad x \in U_{HLie}(\mathfrak{g})
\end{equation*}
for $f,g \in U_{HLie}(\mathfrak{g})^\circ $.
\end{corollary}

\begin{proof}
Let $\{x_1,\dotsc ,x_n\}$ be a basis of $\mathfrak{g}$. Then $U_{HLie}(\mathfrak{g})$ is an almost normalizing extension over $\K$ with generators $x_1, \dotsc , x_n$. Thus the result follows from Theorem \ref{thmHopfDual}.
\end{proof}

\section{Deformation theory and Quantization} \label{def-Poisson}

Deformation is one of the oldest techniques used by mathematicians and physicists. The first instances of the
so-called deformation theory were given by Kodaira and Spencer for complex structures and by Gerstenhaber for
associative algebras \cite{Ge}. The Lie algebras case was studied by Nijenhuis and Richardson \cite{NiRi} and the deformation theory for bialgebras and Hopf algebras were developed later by Gerstenhaber and Schack \cite{GerstShack}.
The main and popular tool is the power series ring or more generally any commutative algebras. By standard facts of deformation
theory, the infinitesimal deformations of an algebra of a given type are parametrized by a second cohomology of the
algebra. More generally, it is stated that deformations are controlled by a suitable cohomology.
A modern approach, essentially due to Quillen, Deligne, Drinfeld, and Kontsevich, is that, in characteristic
zero, every deformation problem is controlled by a differential graded Lie algebra, via solutions of Maurer-Cartan
equation modulo gauge equivalence.

Some mathematical formulations of quantization are based on the algebra of observables and consist in replacing
the classical algebra of observables (typically complex-valued smooth functions on a Poisson manifold) by a
noncommutative one constructed by means of an algebraic formal deformation of the classical algebra. The so-called deformation quantization problem was introduced in the seminal paper \cite{BFFLS} by Bayen, Flato, Fronsdal, Lichnerowicz and Sternheimer  (1978).

In 1997, Kontsevich solved a longstanding problem in mathematical physics, that is every Poisson manifold admits formal
quantization which is canonical up to a certain equivalence.

\subsection{Formal deformation of Hom-associative algebras}
In \cite{MakSil-def}  the formal deformation theory is extended  to Hom-associative and Hom-Lie algebras and is provided a suitable  cohomology. The usual result involving  deformation first order element and second order cohomology groups extends in the Hom case. We describe briefly the results in this section. We consider deformations of Hom-associative algebras  where the twist map remains unchanged.

Let $A = (A, \mu_0,\alpha)$ be a Hom-associative algebra. Let $\K[[t]]$ be the power series ring in one variable $t$ and coefficients in $\K$ and $A[[t]]$ be the set of formal power series whose coefficients are elements of $A$, ($A[[t]]$ is obtained by extending the coefficients domain of $A$ from $\K$ to $\K[[t]]$). Then $A[[t]]$ is a $\K[[t]]$-module. When $A$ is finite-dimensional,
we have $A[[t]] = A \otimes \K[[t]]$. Note that $A$ is a submodule of $A[[t]]$.

\begin{definition}
Let $A = (A, \mu_0,\alpha)$ be a Hom-associative algebra. A formal Hom-associative deformation of $A$ is given by a $\K[[t]]$-bilinear  map $\mu_t : A[[t]] \otimes A[[t]] \to A[[t]]$ of the form
\begin{equation}
\mu_t = \sum_{i \geqslant 0} \mu_i t^i 
\end{equation}
where each $\mu_i$ is a $\K$-bilinear map $\mu_i : A \otimes A \to A$ (extended to be $\K[[t]]$-bilinear) such that holds the following formal Hom-associativity condition:
\begin{equation} \label{A-def_hom-ass}
\as_{\mu_t,\alpha} = \mu_t \circ (\mu_t \otimes \alpha - \alpha \otimes \mu_t) = 0.
\end{equation}
\end{definition}

If $\alpha = Id$ the definition reduces to  formal deformations of an associative algebra.

The equation \eqref{A-def_hom-ass} can be written
\begin{equation}
\sum_{i\in \N} \sum_{j \in \N}  \big( \mu_i(\alpha(x),\mu_j(y,z)) - \mu_i(\mu_j(x,y,\alpha(z)) \big) t^{i+j} = 0.
\end{equation}

Introducing the following notation
\begin{equation*}
(x,y,z) \mapsto \mu_i \circ_\alpha \mu_j(x,y,z) := \mu_i(\alpha(x),\mu_j(y,z)) - \mu_i(\mu_j(x,y,\alpha(z)),
\end{equation*}
the deformation equation may be written as follows
\begin{equation}
\sum_{i\in \N} \sum_{j \in \N}  (\mu_i \circ_{\alpha} \mu_j) t^{i+j} = 0 \qquad \text{or} \qquad \sum_{s \in \N} t^s \sum_{i=0}^{s} \mu_i \circ_{\alpha} \mu_{s-i} = 0.
\end{equation}

It is equivalent to the infinite system: $ \sum_{i=0}^{s} \mu_i \circ_{\alpha} \mu_{s-i} = 0, \  s=0,1,\dotsc .$

The $A$-valued Hochschild Type cohomology of Hom-associative algebras  initiated in \cite{MakSil-def} and extended in \cite{AEM} suits and leads to the following cohomological interpretation:
\begin{enumerate}
\item There is a natural bijection between $H^2(A,A)$ and the set of equivalence classes of deformation (mod $t^2$) of $A$.
\item If $H^2(A,A)=0$ then every deformation of $A$ is trivial.
\end{enumerate}

The fact that the antisymmetrization of the first order element of a deformation of an associative algebra defines a Poisson bracket remains true in the Hom setting. More precisely, we have the following theorem.

\begin{theorem}[\cite{MakSil-def}]
Let $A = (A, \mu_0, \alpha)$ be a commutative Hom-associative algebra and $A_t = (A, \mu_t, \alpha)$ be a
deformation of $A$. Consider the bracket defined for $x,y \in A$ by $\{x,y\} = \mu_1(x,y) - \mu_1(y,x)$ where $\mu_1$ is the first order element of the deformation $\mu_t$. Then $(A, \mu_0, \{~,~\}, \alpha)$ is a Hom-Poisson algebra.
\end{theorem}

The proof is mainly computational, it leans on the properties of the $\alpha$-associators, and on rewriting the deformation equations in terms of coboundary operators.

\subsection{Deformations of Hom-coalgebras and Hom-Bialgebras}
The formal deformation theory  for bialgebras and Hopf algebras was introduced in \cite{GerstShack}. It is extended to  Hom-coalgebras, Hom-bialgebras and Hom-Hopf algebras  in \cite{DekkarMakhlouf}, where a suitable cohomology  is obtained and the classical results are extended to Hom-setting.

\begin{definition}
Let $\left( A,\Delta,\alpha\right)$ be a Hom-coalgebra. A formal
Hom-coalgebra deformation of $A$ is given by a linear map $\Delta_{t}$ :
$A[[t]] \to A[[t]] \otimes A[[t]]$ such that $\Delta_{t}=
{\displaystyle\sum\limits_{i \geqslant 0}}
\Delta_{i}t^{i}$ where each $\Delta_{i}$ is a linear map $\Delta_{i}:A\to A\otimes A$ (extended to be $\K[[t]]$-linear) such that
holds the following formal Hom-coassociativity condition:
\begin{equation}
\left( \Delta_t \otimes \alpha  - \alpha \otimes \Delta_t \right) \circ \Delta_t = 0.
\end{equation}
\end{definition}

\begin{definition}
Let $\left(  A,\mu, \Delta,\alpha\right)$ be a Hom-bialgebra. A formal Hom-bialgebra deformation of $A$ is given by linear maps $\mu_{t} : A[[t]] \otimes A[[t]] \to A[[t]]$ and  $\Delta_{t} : A[[t]] \to A[[t]] \otimes A[[t]]$ such that
\begin{enumerate}
\item $\left(A[[t]],\mu_t,\alpha\right)$ is a Hom-associative algebra,
\item $\left(A[[t]],\Delta_t,\alpha\right)$ is a Hom-coassociative coalgebra,
\item The multiplication and the comultiplication are compatible, that is
\[
\Delta_t \circ \mu_t = \mu_t \otimes \mu_t \circ \tau_{23} \circ (\Delta_t \otimes \Delta_t).
\]
\end{enumerate}
\end{definition}

It is shown in \cite{DekkarMakhlouf} that  deformations are controlled  by Hochschild type cohomology and any deformation of  unital Hom-associative algebra (resp. counital Hom-coassociative colgebra) is equivalent to unital Hom-associative algebra (resp. counital Hom-coassociative colgebra). Furthermore, any deformation of a Hom-Hopf algebra as a Hom-bialgebra is automatically a Hom-Hopf algebra.

In a similar way as for Hom-associative algebra, we have:

\begin{theorem}
Let $A = (A, \Delta_0, \alpha)$ be a cocommutative Hom-coassociative coalgebra and $A_t = (A, \Delta_t, \alpha)$ be a
deformation of $A$. Consider the cobracket defined for $x \in A$ by $\delta (x) = \Delta_1(x) - \Delta^{op}_1(x)$ where $\Delta_1$ is the first order element of the deformation $\Delta_t$. Then $(A, \Delta_0, \delta, \alpha)$ is a Hom-coPoisson algebra.
\end{theorem}
\subsection{Quantization and Twisting star-products}

The deformation quantization problem in the Hom-setting is stated as follows: given a Hom-Poisson algebra (resp. Hom-coPoisson algebra), find a formal deformation of a commutative Hom-associative algebra (resp. cocommutative Hom-coassociative coalgebra) such that the first order element of the deformation defines the Hom-Poisson algebra (resp. Hom-coPoisson algebra). In the classical case the Hom-Poisson structure is called the quasi-classical limit and the deformation is the star-product. This point of view initiated in \cite{BFFLS} attempts to view the quantum mechanics as a deformation of the classical mechanics, the Lorentz group is a deformation of the Galilee group.


%


Let $(A, \cdot,\{~,~\},\alpha)$ be a commutative Hom-associative algebra endowed with a Hom-Poisson barcket $\{~,~\}$.

\begin{definition}
A $\star$-product on $A$ is a one parameter formal deformation defined on $A$ by 
\[
f \star_{t} g = \sum_{r=0}^{\infty}t ^r \mu_{r}(f,g)
\]
such that
\begin{enumerate}
\item The $\star$-product in $A[[t]]$ is Hom-associative, that is
\[
\forall r\in \mathbb{N},\quad \sum_{i=0}^{r}(\mu_{i}(\mu_{r-i}(f,g),\alpha(h))-(\mu_{i}(\alpha(f),\mu_{r-i}(g,h)))=0,
\]
\item \label{initial_product} $\mu_{0}(f,g)=f \cdot g$
\item $\mu_{1}(f,g)-\mu_1(g,f)=\{f,g\}$
\item $\mu_{r}(f,1) = \mu_{r}(1,f) = 0$ $\quad \forall\ r>0$
\end{enumerate}
\end{definition}

\begin{remark} \hfill
\begin{itemize}
\item The condition \eqref{initial_product} shows that $[f,g]:= \frac{1}{2t}\left(f \star_{t} g\; - \; g\star_{t}f\right)$ is a deformation of the Hom-Lie structure $\{~,~\}$.
\item The condition $\mu_{1}(f,g)-\mu_1(g,f)=\{f,g\}$ expresses the correspondence between the deformation and the Hom-Poisson structure
%
\[
\frac{f \star_{t}g-g \star_{t}f}{t}|_{t =0}=\{f,g\}.
\]
\end{itemize}
\end{remark}

Similarly we set the dual version of quantization problem as follows.

Let $A$ be cocommutative Hom-coPoisson bialgebra (resp. Hopf algebra) and let $\delta$ be its Poisson cobracket. A quantization of $A$ is a Hom-bialgebra (resp. Hom-Hopf algebra) deformation $A_t$ of $A$  such that
\[
\delta (x)=\frac{\Delta_t (a)-\Delta_t^{op}(a)}{t} \pmod t,
\]
where $x\in A$ and $a$ is any element of $A[[t]]$ such that $x=a \pmod t$.

\begin{theorem}
Let $(A, \star)$ be an associative deformation of an associative algebra $(A,\mu_0)$, with $\star = \mu_0 + t \mu_1 + t^2 \mu_2 + \dotsb$ \\
Let $\alpha : A \to A$ be a morphism such that for all $i \in \N,\ \alpha \circ \mu_i = \mu_i \circ \alpha^{\otimes 2}$. Then $(A, \star_\alpha = \alpha \circ \star, \alpha)$ is a Hom-associative deformation of $A$.
\end{theorem}

\begin{proof}
Since for all $i \in \N,\ \alpha \circ \mu_i = \mu_i \circ \alpha^{\otimes 2}$, we also have $\alpha \circ \star = \star \circ \alpha^{\otimes 2}$. For $f,g,h \in (A, \star_\alpha, \alpha)$, we get
\begin{equation}
\begin{split}
&(f \star_\alpha g) \star_\alpha \alpha(h) = \alpha(\alpha(f) \star \alpha(g)) \star \alpha(\alpha(h)) = \alpha((\alpha(f) \star \alpha(g)) \star \alpha(h)) \\
&= \alpha(\alpha(f) \star (\alpha(g) \star \alpha(h))) = \alpha(\alpha(f)) \star \alpha(\alpha(g) \star \alpha(h)) = \alpha(f) \star_\alpha (g \star_\alpha h),
\end{split}
\end{equation}
the passage from the first line to the second is due to the associativity of $\star$. This shows that the product $\star_\alpha$ is Hom-associative.
\end{proof}

\subsection{Moyal-Weyl Hom-associative algebra}
In the following, we twist the Moyal-Weyl product. It is the associative $\star$-product corresponding to the deformation of the Poisson phase-space bracket, one of the first examples of Kontsevich formal deformation \cite{Kont}.

We consider the Poisson algebra of polynomials of two variables $(\R[x,y],\cdot,\{~,~\})$ where the Poisson bracket of two polynomials is given by $\{f,g\} = \dfrac{\partial f}{\partial x}\dfrac{\partial g}{\partial y} - \dfrac{\partial f}{\partial y}\dfrac{\partial g}{\partial x}$.

The associated associative algebra is $(\R[x,y], \star)$ where the star-product is given by the Moyal-Weyl formula
\begin{equation}
f \star g = \sum_{n \in \N} \frac{\partial^n f}{\partial x^n} \frac{\partial^n g}{\partial y^n} \frac{t^n}{n!} = \sum_{n \in \N} \mu_n(f,g) t^n,
\end{equation}
where $\mu_n(f,g)=\dfrac{1}{n!} \dfrac{\partial^n f}{\partial x^n} \dfrac{\partial^n g}{\partial y^n}.$

\begin{proposition} \label{caract_alpha_ass}
A morphism $\alpha : \R[x,y] \to \R[x,y]$ satisfying $\alpha \circ \mu_i = \mu_i \circ \alpha^{\otimes 2}$ for all $i \in \N$ which gives $(\R[x,y],\ \star_\alpha = \alpha \star, \alpha)$ a structure of Hom-associative algebra is of the form
\begin{equation}
\alpha(x) = ax+b \ \textrm{and}\ \alpha(y) = \frac{1}{a}y+c \qquad \textrm{where} \quad a,b,c \in \R,\ a \neq 0.
\end{equation}
\end{proposition}

\begin{proof}
Let $\alpha : \R[x,y] \to \R[x,y]$ be a morphism such that for all $i \in \N,\ \alpha \mu_i = \mu_i \alpha^{\otimes 2}$. In particular, for $i=0$,
\begin{equation}
\alpha(fg) = \alpha(f) \alpha(g),
\end{equation}
which shows that $\alpha$ is multiplicative, so it is sufficient to define it on $x$ and $y$. For $i=1$, we get
\begin{equation} \label{mu_1}
\alpha \left(\dfrac{\partial f}{\partial x} \dfrac{\partial g}{\partial y} \right) = \dfrac{\partial \alpha(f)}{\partial x} \dfrac{\partial \alpha(g)}{\partial y},
\end{equation}
which implies that $\alpha(\{f,g\}) = \{\alpha(f),\alpha(g)\}$.

We set $P_1(x,y) := \alpha(x)$ and $P_2 := \alpha(y)$. For $f(x,y)=x$ and $g(x,y)=y$, the equation \eqref{mu_1} gives
\begin{equation*}
1 = \alpha(1) = \frac{\partial P_1}{\partial x} \frac{\partial P_2}{\partial y},
\end{equation*}
so we must have $P_1(x,y) = ax+Q_1(y)$ and $P_2(x,y) = \dfrac{1}{a}y + Q_2(x)$ with $a \in \R \setminus \{0\}$ and $Q_1,Q_2 \in \R[x,y]$. For $f(x,y)=x$ and $g(x,y)=y$, the equation \eqref{mu_1} gives
\begin{equation*}
0 = \alpha(0) = \frac{\partial P_2}{\partial x} \frac{\partial P_1}{\partial y} = Q'_{2,x} Q'_{1,y}.
\end{equation*}
So we can suppose that $Q_1(y) = b$ is constant. Reporting in the equation \eqref{mu_1}, with $f(x,y)=g(x,y)=y$, we find
\begin{equation*}
0 = \alpha(0) = \frac{\partial P_2}{\partial x} \frac{\partial P_2}{\partial y} = Q'_{2,x} \dfrac{1}{a},
\end{equation*}
so $Q_2(x) = c$ is constant.
It remains to verify that for $i > 1$, $\alpha \mu_i = \mu_i \alpha^{\otimes 2}$ \textit{i.e.}\/ for $f,g \in \R[x,y],\ \alpha \left(\dfrac{\partial^i f}{\partial x^i} \dfrac{\partial^i g}{\partial y^i} \dfrac{1}{i!} \right) = \dfrac{\partial^i \alpha(f)}{\partial x^i} \dfrac{\partial^i \alpha(g)}{\partial y^i} \dfrac{1}{i!}$. By multiplicavity of $\alpha$, the only non trivial case is $f(x,y) = x^n$ and $g(x,y)=y^m$. We have
\begin{equation*}
\begin{split}
& \alpha \left(\dfrac{\partial^i f}{\partial x^i} \dfrac{\partial^i g}{\partial y^i} \frac{1}{i!} \right) = \alpha \left( i! \binom{n}{i} x^{n-i} i! \binom{m}{i} y^{m-i} \frac{1}{i!} \right) = i! \binom{n}{i} i! \binom{m}{i} (ax+b)^{n-i} \left(\dfrac{1}{a}y+c\right)^{m-i} \frac{1}{i!} \\
&= i! \binom{n}{i} i! \binom{m}{i} (ax+b)^{n-i}a^i \left(\dfrac{1}{a}y+c\right)^{m-i}\left(\dfrac{1}{a}\right)^i \frac{1}{i!} = \dfrac{\partial^i (ax+b)^n}{\partial x^i} \dfrac{\partial^i \left(\dfrac{1}{a}y+c\right)^m}{\partial y^i} \frac{1}{i!} \\
&= \dfrac{\partial^i \alpha(f)}{\partial x^i} \dfrac{\partial^i \alpha(g)}{\partial y^i} \frac{1}{i!}.
\end{split}
\end{equation*}
\end{proof}

The Hom-algebra $(\R[x,y], \star_\alpha, \alpha)$ is Hom-associative and not associative if $\alpha \neq Id$. Indeed, for $f(x,y) = 1,\ g(x,y) = y$ and $h(x,y) = x$, we have
\begin{gather*}
(f \star_\alpha g) \star_\alpha h = \alpha(\alpha(f) \star \alpha(g)) \star \alpha(h) \\
= \alpha\left(1 \star \frac{1}{a}y+c\right) \star (ax+b) = \alpha\left(\frac{1}{a}y+c\right) \star (ax+b), \\
\intertext{and}
f \star_\alpha (g \star_\alpha h) = \alpha(\alpha(f)) \star \alpha(\alpha(g) \star \alpha(h)) \\
= 1 \star \alpha\left(\left(\frac{1}{a}y+c\right) \star (ax+b) \right) = \alpha\left(\frac{1}{a}y+c\right) \star \alpha(ax+b)
\end{gather*}
which are different in general.

More generally, we can consider the Poisson algebra of polynomials of $n \geqslant 3$ variables $(\R[x_1,\dotsc ,x_n],\cdot,\{~,~\})$ where the Poisson bracket of two polynomials is given by $\displaystyle \{f,g\} = \sum_{1\leqslant i,j \leqslant n} \tau_{ij} \frac{\partial f}{\partial x_i} \frac{\partial g}{\partial x_j}$, with $\tau = (\tau_{ij})$ an antisymmetric $n \times n$ real matrix. \\

The associated associative algebra is $(\R[x_1\dotsc ,x_n], \star)$ where the star-product is given by the Moyal-Weyl formula
\begin{equation}
f \star g = \sum_{n \in \N} \sum_{1\leqslant i_1,j_1,\dotsc ,i_n,j_n\leqslant n} \sigma_{i_1 j_1} \dotsm \sigma_{i_n j_n} \frac{\partial^n f}{\partial x_{i_1} \dotsm \partial x_{i_n}} \frac{\partial^n g}{\partial x_{j_1} \dotsm \partial x_{j_n}} \frac{t^n}{n!},
\end{equation}
where $\sigma = (\sigma_{ij})$ is the matrix whose antisymmetrization is $\tau$.

Set $n>2$ and $\mu_n= \dfrac{\partial^n f}{\partial x_{i_1} \dotsm \partial x_{i_n}} \dfrac{\partial^n g}{\partial x_{j_1} \dotsm \partial x_{j_n}}.$

\begin{proposition}
The only morphisms $\alpha : \R[x_1,\dotsc ,x_n] \to \R[x_1,\dotsc ,x_n]$ satisfying $\alpha \circ \mu_i = \mu_i \circ \alpha^{\otimes 2}$ for all $i \in \N$ which give $(\R[x_1,\dotsc ,x_n],\ \star_\alpha = \alpha \star, \alpha)$ a structure of Hom-associative algebra are of the form
\begin{equation}
\forall\ 1\leqslant i \leqslant n,\ \alpha(x_i) = x_i+b_i \ \quad \textrm{or} \quad \forall\ 1\leqslant i \leqslant n,\ \alpha(x_i) = -x_i+b_i \qquad \textrm{where} \quad b_i \in \R .
\end{equation}
\end{proposition}

\begin{proof}
The proof is similar to the case with two variables, we get $\alpha(x_i) = a_i x_i + b_i$, except that this time, $a_i a_j = 1$ for all $i \neq j$, which gives the two cases of the proposition.
\end{proof}

\subsection{Moyal-Weyl Hom-Poisson algebra}

We consider the Poisson algebra of polynomials of two variables $(\R[x,y],\cdot,\{~,~\})$ where the Poisson bracket of two polynomials is given by $\{f,g\} = \dfrac{\partial f}{\partial x}\dfrac{\partial g}{\partial y} - \dfrac{\partial f}{\partial y}\dfrac{\partial g}{\partial x}$.

\begin{proposition} \label{MW_Hom-Poisson}
A morphism $\alpha : \R[x,y] \to \R[x,y]$ which gives $(\R[x,y],\quad \cdot_\alpha = \alpha \circ \cdot,\quad \{~,~\}_\alpha = \alpha \circ \{~,~\},\quad \alpha)$ a structure of Hom-Poisson algebra satisfies the equation
\begin{equation} \label{eq_alpha}
1 = \frac{\partial \alpha(x)}{\partial x} \frac{\partial \alpha(y)}{\partial y} - \frac{\partial \alpha(x)}{\partial y} \frac{\partial \alpha(y)}{\partial x}.
\end{equation}
\end{proposition}

\begin{proof}
Since $\alpha$ is a morphism of Poisson algebra, it satisfies, for all $f,g \in \R[x,y]$
\begin{gather*}
\alpha(f \cdot g) = \alpha(f) \cdot \alpha(g) \\
\alpha(\{f,g\} = \{\alpha(f),\alpha(g)\}.
\end{gather*}

The first equation shows that it is sufficient to define $\alpha$ on $x$ and $y$. For the second equation, we can suppose by linearity that $f(x,y) = x^k y^l$ and $g(x,y) = x^p y^q$. It then rewrites
\begin{equation*}
\begin{split}
kq \alpha(x)^{k-1} \alpha(y)^l \alpha(x)^p \alpha(y)^{q-1} - pl \alpha(x)^k \alpha(y)^{l-1} \alpha(x)^{p-1} \alpha(y)^q \\
= \frac{\partial (\alpha(x)^k \alpha(y)^l)}{\partial x} \frac{\partial (\alpha(x)^p \alpha(y)^q)}{\partial y} - \frac{\partial (\alpha(x)^k \alpha(y)^l)}{\partial y} \frac{\partial (\alpha(x)^p \alpha(y)^q)}{\partial x}
\end{split}
\end{equation*}
which simplifies in
\begin{equation*}
1 = \frac{\partial \alpha(x)}{\partial x} \frac{\partial \alpha(y)}{\partial y} - \frac{\partial \alpha(x)}{\partial y} \frac{\partial \alpha(y)}{\partial x}.
\end{equation*}
\end{proof}

\begin{example}
We give some examples for the morphism $\alpha$. We set 
\begin{gather*}
\alpha(x) = \Gamma_1(x,y) = \sum_{0 \leqslant i,j \leqslant d} a_{ij} x^i y^j \\
\alpha(y) = \Gamma_2(x,y) = \sum_{0 \leqslant i,j \leqslant d} b_{ij} x^i y^j \\
\end{gather*}
with $\Gamma_1,\Gamma_2 \in \R[x,y]$, and $d$ the bigger degree for each variable. Since the equation \eqref{eq_alpha} only contains derivatives of $\Gamma_1$ and $\Gamma_2$, without loss of generality we assume $a_{00} = b_{00} = 0$.

\begin{description}
\item[Degree 1]
\begin{gather*}
\Gamma_1(x,y) = a_{10} x + a_{01} y \\
\Gamma_2(x,y) = b_{10} x + b_{01} y
\end{gather*}
The equation \eqref{eq_alpha} becomes
\begin{equation} \label{eq_alpha_deg1}
1 = a_{10} b_{01} - a_{01} b_{10}.
\end{equation}
The polynomials $\Gamma_1,\Gamma_2$ are of one of the following form
\begin{enumerate}[(i)]
\item $\Gamma_1(x,y) = a_{10} x + a_{01} y, \qquad \Gamma_2(x,y) = -\frac{1}{a_{10}}x \qquad \textrm{with}\ a_{10} \neq 0$ \label{sol_deg1_1},
\item $\Gamma_1(x,y) = \frac{1+a_{01} b_{10}}{b_{01}} x + a_{01} y, \qquad \Gamma_2(x,y) = b_{10} x + b_{01} y \qquad \textrm{with}\ b_{01} \neq 0$. \label{sol_deg1_2}
\end{enumerate}

\item[Degree 2] For simplicity, we only take one of the polynomials of degree two.
\begin{gather*}
\Gamma_1(x,y) = a_{10} x + a_{01} y \\
\Gamma_2(x,y) = b_{10} x + b_{01} y + b_{20} x^2 + b_{11} xy + b_{20} y^2 
\end{gather*}
Arranging the equation \eqref{eq_alpha} by degree, we obtain the following system of equations.
\begin{equation*}
\left\{
\begin{aligned}
1 &= a_{10} b_{01} - a_{01} b_{10} \\
0 &= 2 a_{10} b_{02} - a_{01} b_{11} \\
0 &= a_{10} b_{11} - 2 a_{01} b_{20} \\
\end{aligned}
\right.
\end{equation*}
The polynomials $\Gamma_1,\Gamma_2$ are of one of the following form
\begin{enumerate}[(i)]
\item $\Gamma_1(x,y) = \frac{1+a_{01} b_{10}}{b_{01}} x + a_{01} y, \qquad \Gamma_2(x,y) = b_{10} x + b_{01} y$
\item $\Gamma_1(x,y) = \frac{1}{b_{01}} x, \qquad \Gamma_2(x,y) = b_{20} x^2 + b_{10} x + b_{01} y$
\item $\Gamma_1(x,y) = a_{10} x + \frac{2a_{10} b_{02}}{b_{11}} y, \qquad \Gamma_2(x,y) = \frac{1}{a_{01}} y + \frac{b_{11}^2}{4 b_{02}} x^2 + b_{11} xy + b_{02} y^2$
\item $\Gamma_1(x,y) = a_{10} x + \frac{2 a_{10} b_{02}}{b_{11}} y, \qquad \Gamma_2(x,y) = b_{10} x + \frac{2 a_{10} b_{02} b_{10} + b_{11}}{a_{10} b_{11}} y + \frac{b_{11}^2}{4 b_{02}} x^2 + b_{11} x y + b_{02} y^2$
\end{enumerate}
\end{description}
\end{example}

We now want to deform the classical Moyal-Weyl star-product on $\R[x,y]$ using morphisms $\alpha$ previously found.
For $P,Q \in \R[x,y]$ we define
\begin{equation}
P \star_\alpha Q = \sum_{n \geqslant 0} {\mu_n}_\alpha(P,Q) t^n \quad \textrm{with} \quad {\mu_n}_\alpha(P,Q) = \frac{\partial \alpha(P)}{\partial x_n} \frac{\partial \alpha(Q)}{\partial y_n} \frac{1}{n!}.
\end{equation}

The Hom-associator writes
\begin{equation*}
\begin{split}
\mathfrak{as}_{\star_\alpha,\alpha}(P,Q,R) &= \sum_{p \geqslant 1} \sum_{n=0}^p \big( {\mu_{p-n}}_\alpha({\mu_n}_\alpha(P,Q),\alpha(R) - {\mu_{p-n}}_\alpha(\alpha(P),{\mu_n}_\alpha(Q,R)) \big) t^p \\
&= \sum_{p \geqslant 1} \sum_{n=0}^p \frac{1}{n!(p-n)!} \left( \frac{\partial \alpha\left( \frac{\partial \alpha(P)}{\partial x^n} \frac{\partial \alpha(Q)}{\partial y^n} \right)}{\partial x^{p-n}} \frac{\partial \alpha^2(R)}{\partial y^{p-n}} - \frac{\partial \alpha^2(P)}{\partial x^{p-n}} \frac{\partial \alpha\left( \frac{\partial \alpha(Q)}{\partial x^n} \frac{\partial \alpha(R)}{\partial y^n} \right)}{\partial y^{p-n}} \right),
\end{split}
\end{equation*}
since the multiplication $\cdot_\alpha$ is Hom-associative, the constant term vanishes. \\

Trying to make the coefficient in $t^n$ vanish for particular polynomials $P,Q,R$, we obtain conditions on the morphism $\alpha$.

\subsubsection{\texorpdfstring{Degree $1$ case}{Degree 1 case}}

In degree $1$, the solutions of \eqref{eq_alpha_deg1} are of the form \eqref{sol_deg1_1} or \eqref{sol_deg1_2}. \\

In the first case \eqref{sol_deg1_1}, the morphism $\alpha$ satisfies
\begin{gather*}
\alpha(x) = \Gamma_1(x,y) = a_{10} x + a_{01} y, \\
\alpha(y) = \Gamma_2(x,y) = -\frac{1}{a_{01}}x,
\end{gather*}
with $a_{10} \neq 0$. \\
For the particuliar polynomials $P(x,y) = x, Q(x,y) = y = R(x,y)$, we have
\begin{equation*}
\mathfrak{as}_{\star_\alpha,\alpha}(P,Q,R) = y \neq 0
\end{equation*}
so $\mathfrak{as}_{\star_\alpha,\alpha} \neq 0$ and this morphism $\alpha$ does not give a deformation of the Moyal-Weyl star-product. \\

In the second case \eqref{sol_deg1_2} the morphism $\alpha$ satisfies
\begin{gather*}
\alpha(x) = \Gamma_1(x,y) = \frac{1+a_{01} b_{10}}{b_{01}} x + a_{01} y, \\
\alpha(y) = \Gamma_2(x,y) = b_{10} x + b_{01} y,
\end{gather*}
with $b_{01} \neq 0$. \\
Renaming $c:= a_{01}b_{10}$ if $b_{10} \neq 0$, we can write
\begin{gather*}
\alpha(x) = \Gamma_1(x,y) = \frac{1+c}{b_{01}} x + \frac{c}{b_{10}} y, \\
\alpha(y) = \Gamma_2(x,y) = b_{10} x + b_{01} y.
\end{gather*}
For the particuliar polynomials $P(x,y) = x, Q(x,y) = y = R(x,y)$, the coefficient in $t$ of $\mathfrak{as}_{\star_\alpha,\alpha}(P,Q,R)$ vanishes only if $c=0$ and $b_{01} = 0$ or if $c=0$ and $b_{10} = 0$, which is not possible. \\
If $b_{10} = 0$ then the morphism $\alpha$ satisfies
\begin{gather*}
\alpha(x) = \Gamma_1(x,y) = \frac{1}{b_{01}} x + a_{01} y, \\
\alpha(y) = \Gamma_2(x,y) = b_{01} y,
\end{gather*}
and the coefficient in $t$ of $\mathfrak{as}_{\star_\alpha,\alpha}(P,Q,R)$ vanishes only if $a_{01} = 0$. In that case, the morphism $\alpha$ is as in the Proposition \ref{caract_alpha_ass}, and thus gives a deformation of the Moyal-Weyl star-product.

Finally, the only morphisms $\alpha$ of degree $1$ which give raise to a deformation of the Moyal-Weyl star-product are as in the Proposition \ref{caract_alpha_ass}.

\subsubsection{\texorpdfstring{Degree $2$ case}{Degree 2 case}}

In degree $2$, computations done with the computer algebra system \textit{Mathematica} also led to the case given by the Proposition \ref{caract_alpha_ass}.

\vspace{1cm}

We conjecture that the only morphisms $\alpha$ of the Proposition \ref{MW_Hom-Poisson} are the one found in the Proposition \ref{caract_alpha_ass}.

An interseting question would be to know if twisting by a morphism $\alpha$ is functorial. If this is the case, the previous conjecture would be true and we would have the following commutative diagram
\begin{equation}
\xymatrix{
(A,\mu,\{~,~\}) \ar[r]^-{\text{twist}_\alpha} \ar[d] & (A,\mu_\alpha,\{~,~\}_\alpha,\alpha) \ar[d] \\
(A,\star) \ar[u]^{\text{Kont}} \ar[r]_{\text{twist}_\alpha} & (A,\star_\alpha,\alpha) \ar[u]_{\text{Kont}}
}
\end{equation}
\newline
with $\text{Kont}$ the Kontsevich bijection between the set of equivalence classes of Poisson brackets on the commutative $\K[[t]]$-algebra $A[[t]]$ and the set of equivalence classes of star products, and $\text{twist}_\alpha : \textsf{Poiss} \to \textsf{Hom-Poiss}$ the functor from the category \textsf{Poiss} of Poisson algebras to the category \textsf{Hom-Poiss} of Hom-Poisson algebras. More generally, $\text{twist}_\alpha$ is a functor from \textsf{Struct} to \textsf{Hom-Struct}, where \textsf{Struct} is a category of structures such as associative algebras \textsf{Ass}, Lie algebras \textsf{Lie}, Poisson algebra \textsf{Poiss}, and so on, and \textsf{Hom-Struct} the corresponding category of Hom-structures.

\bibliographystyle{amsplain}
\providecommand{\bysame}{\leavevmode\hbox to3em{\hrulefill}\thinspace}
\providecommand{\MR}{\relax\ifhmode\unskip\space\fi MR }
\providecommand{\MRhref}[2]{%
  \href{http://www.ams.org/mathscinet-getitem?mr=#1}{#2}
}
\providecommand{\href}[2]{#2}


\noindent 
       Martin Bordemann, Olivier Elchinger and Abdenacer Makhlouf \\
      Universit\'e de Haute Alsace\\
      Laboratoire de Math\'ematique, Informatique et Applications\\
      4, rue des Fr\`eres Lumi\`ere, 68093 Mulhouse cedex France\\
    Martin.Bordemann@uha.fr\\
    Olivier.Elchinger@uha.fr\\
    Abdenacer.Makhlouf@uha.fr

\label{lastpage}
\end{document}